\documentclass{article}
\usepackage[T1]{fontenc}
\usepackage[english]{babel}
\usepackage{amsmath, amssymb, amsthm}
\usepackage{mathrsfs}
\usepackage{geometry}
\usepackage{graphicx}
\usepackage{fancyhdr} 
\usepackage{tikz}
\usetikzlibrary{decorations.markings,arrows.meta}
\usepackage{amscd}
\usepackage{epsf}
\usepackage{booktabs}
\usepackage{array}
\usepackage{multirow}
\usepackage{color}
\usepackage[symbol]{footmisc}
\usepackage{lipsum}
\usepackage{latexsym}
\usepackage{verbatim}
\usepackage{bbm}
\usepackage{setspace}
\usepackage[backref=page]{hyperref}
\renewcommand*{\backref}[1]{}
\renewcommand*{\backrefalt}[4]{%
	\ifcase #1 (Not cited)%
	\or        (Cited on page~#2)%
	\else      (Cited on pages~#2)%
	\fi}
\usepackage{cleveref}
\usepackage{etoolbox}
\usepackage[normalem]{ulem} 
\usepackage{thmtools} 
\usepackage{enumitem}
\usepackage{anyfontsize}

\setlist[enumerate,1]{label=\textbf{\arabic*.}, ref=\arabic*, leftmargin=*, itemsep=0.5ex, topsep=0.5ex}

\let\OLDthebibliography\thebibliography
\renewcommand\thebibliography[1]{
  \OLDthebibliography{#1}
  \setlength{\parskip}{2pt}
  \setlength{\itemsep}{0pt plus 0.3ex}
}

\makeatletter
\renewenvironment{proof}[1][\proofname]{%
  \par\pushQED{\qed}%
  \normalfont\topsep6\p@\@plus6\p@\relax
  \trivlist
  \item[\hskip\labelsep
        \textbf{\textup{#1}}\@addpunct{.}]\ignorespaces
}{%
  \popQED\endtrivlist\@endpefalse
}
\makeatother
\newcommand{\mysubjclass}{\textup{2020} \textbf{Mathematics Subject Classification:} 53C15, 53C25, 53C29, 53C55.}
\newcommand{\mykeywords}{\textbf{Keywords:} Almost contact metric structures;
metric connections with skew-symmetric torsion;
Sasaki with torsion;
$\nabla$-Einstein manifolds;
co-K\"ahler-like manifolds; $U(n)$-flows.
}

\fancypagestyle{firstpagefooter}{
  \fancyhf{} 
  \fancyfoot[L]{\scriptsize
    \mysubjclass \\
    \mykeywords \\
    }
}

\hypersetup{
    colorlinks=true,
    linkcolor=blue,    
    citecolor=red,     
    urlcolor=blue,     
    pdftitle={Sasaki with torsion manifolds and string backgrounds},
    pdfauthor={Beatrice Brienza, Anna Fino, Udhav Fowdar, Gueo Grantcharov}
}

\newtheorem{thm}{Theorem}[section]
\newtheorem{prop}[thm]{Proposition}
\newtheorem{lem}[thm]{Lemma}
\newtheorem{cor}[thm]{Corollary}

\newtheorem{defn}[thm]{Definition}

\theoremstyle{definition}
\newtheorem{rmk}[thm]{Remark}
\newtheorem{ex}[thm]{Example}

\title{titolo}

\newcommand{\g}{\mathfrak{g}}
\newcommand{\Ric}{\mathrm{Ric}}
\newcommand{\h}{\mathfrak{h}}

\newcommand{\R}{\mathbb{R}}
\newcommand{\C}{\mathbb{C}}
\newcommand{\SU}{\mathrm{SU}}
\newcommand{\SO}{\mathrm{SO}}

\newcommand{\w}{\wedge}
\newcommand{\vol}{{\mathrm{vol}}}

\renewcommand{\epsilon}{\varepsilon}
\renewcommand{\phi}{\varphi}

\usetikzlibrary{shapes.misc}
\tikzset{cross/.style={cross out, draw, 
         minimum size=2*(#1-\pgflinewidth), 
         inner sep=0pt, outer sep=0pt},
         cross/.default={3.3pt}}

\numberwithin{equation}{section}

\title{\textbf{Sasaki with torsion manifolds and string backgrounds}}
\author{Beatrice Brienza, Anna Fino, Udhav Fowdar, Gueo Grantcharov}
\date{}

\begin{document}

\maketitle

\begin{abstract}
 Motivated by the analogy with the Bismut connection in Hermitian geometry, we study Sasaki with torsion manifolds. In particular, we characterize co-K\"ahler-like and flat Sasaki with torsion manifolds, and we introduce the notion of a $\nabla$-Einstein manifold as the odd-dimensional analogue of the Bismut Hermite-Einstein condition. We provide non-compact examples and we study compact $\nabla$-Einstein manifold in dimension $5$ and $7$. We also develop a general framework for geometric flows of almost contact metric structures. In particular, we derive a flow for Sasaki with torsion structures that preserves the strong condition, that is, the closure of the torsion. Furthermore we prove that such flow is gauge-equivalent to the generalized Ricci flow and it is gauge-equivalent to the pluriclosed flow, after performing a trivial product with $S^1$.

\end{abstract}

\thispagestyle{firstpagefooter}
\tableofcontents

\section{Introduction}

Metric connections with totally skew-symmetric torsion have been extensively studied in the context of (almost) Hermitian geometry, where the Bismut connection plays a central role and a rich theory has been developed, both from the differential-geometric and physical viewpoints (see \cite{Agricola2006,FiGr,GFStr,FI, IP2} and the references therein). The corresponding theory in the setting of almost contact metric geometry is considerably less developed, despite its natural appearance in odd-dimensional analogues of Hermitian structures and its relevance in string theory. In dimension five, the solutions of the type II string equations are indeed related to suitable almost contact metric structures \cite{FI2}. However, not every such geometric structure is admissible. One requires the existence of a connection $\nabla$ with totally skew-symmetric torsion that preserves the almost contact metric structure, together with a $\nabla$-parallel spinor field satisfying the so-called \emph{Killing spinor equations}. A notable case occurs when the holonomy of $\nabla$ reduces to $\mathrm{SU}(2)$, which is equivalent to the vanishing of an analogue of the Bismut--Ricci form (see Section~5). Such a geometry thus provides an example of a string background, here defined as metric connection with (closed) skew torsion and reduced holonomy.

Friedrich and Ivanov showed in  \cite{FI}  that a normal almost contact metric manifold $ (M,\phi,\xi,\eta,g) $ admits a \emph{canonical} metric connection $ \nabla $ with totally skew-symmetric torsion preserving the structure, i.e.  
\[
\nabla \phi = 0, \qquad \nabla \eta = 0,
\]  
if and only if the Reeb vector field $ \xi $ is Killing. In this case, the connection is uniquely determined and can be written as  
\[
g(\nabla_X Y, Z)
= g(\nabla^g_X Y, Z)
+ \frac{1}{2}\bigl(\eta \wedge d\eta + d^\phi F\bigr)(X,Y,Z),
\]  
where $d^\phi F(X,Y,Z) = -dF(\phi X,\phi Y,\phi Z)$. 

This construction provides a natural notion of \emph{Sasaki with torsion geometry}, i.e. the odd-dimensional counterpart of K\"ahler with torsion geometry \cite{IP2}. Examples are abundant: for instance, Sasaki and quasi-Sasaki manifolds, and, more general, $S^1$-bundles over Hermitian manifolds \cite{CoMa}.

As a matter of fact, the connection $\nabla$ is defined under the sole assumption that the Nijenhuis tensor is skew-symmetric \cite{FI}. In dimension five, however, the skew-symmetry of the Nijenhuis tensor is equivalent to its vanishing \cite{CM}; since our primary interest lies in such dimension, we may therefore impose the \emph{normality condition} without loss of generality.
\medskip

The parallel between Sasaki with torsion geometry and K\"ahler with torsion geometry is particularly striking. Therefore, it seems natural to investigate the properties of $\nabla$ that are known to hold for the Bismut connection. For instance, following the notion introduced by Angella, Otal, Ugarte, and Villacampa in the Hermitian case \cite{AOUV}, we call a Sasaki with torsion manifold $(M,\phi,\xi,g,\nabla)$ \emph{co-K\"ahler-like} if $\nabla$ satisfies the first Bianchi identity. In Section 3 (Theorem ~\ref{thm:CKL}) we completely characterize this condition:

\begin{thm} \label{thm:CKL intro}
Let $(M,\phi,\xi,g,\nabla)$ be a Sasaki with torsion manifold. Then $(\phi,\xi,g)$ is co--K\"ahler-like if and only if the torsion $3$-form
\[
H = \eta \wedge d\eta + d^{\phi}F
\]
is both closed and parallel.
\end{thm}

This result will be used in Section~4, where we study the flatness condition for the connection $\nabla$ and establish the Sasaki analogue of the Bismut flat case investigated by Wang, Yang, and Zheng \cite{WYZ}. Theorem~\ref{thm:flat} shows that for $M$ compact, $\nabla$ is flat if and only if, up to a finite cover, $M$ is a quotient $M = G / \mathbb{Z}^k$ of a simply connected odd-dimensional Lie group $G$ equipped with a left-invariant Sasaki with torsion structure whose metric is bi-invariant.

In Section~5 we investigate the curvature $R^\nabla$ of Sasaki with torsion manifolds, focusing on \emph{$\nabla$-Einstein} manifolds. These are defined as Sasaki with torsion manifolds for which the $\nabla$-Ricci form $\rho^\nabla$ vanishes and the torsion $H$ is closed, the odd-dimensional counterpart of the Bismut-Hermite-Einstein condition \cite{GJS2023}.

Both compact and non-compact cases are considered, with special attention to dimension~$5$. We first construct new non-compact examples of $\nabla$-Einstein structures in dimensions $5$ and $7$. Related examples in dimensions five and seven were pointed out to us by
V.~Apostolov \cite{ApostolovPrivate2026}. Turning to the compact case, we obtain a complete classification of compact $5$-dimensional $\nabla$-Einstein manifolds (Theorem~\ref{thm:5}) and we explore the $7$-dimensional case, providing a full characterization \ref{thm:7dimensional-structure}. The key ingredient is a result from \cite{KS}: any compact $\nabla$-Einstein manifold admits a $\nabla$-parallel vector field
\[
V = \theta^{\sharp} - \operatorname{grad} f,
\]
where $\theta$ is the Lee form of the transverse Hermitian structure and $f$ is a suitably normalized smooth function. The existence of such a vector field is ensured by the following two facts: \cite{FI} (see also \cite{KS}) establishes that $(\nabla, g, H, \theta^\sharp)$ is a generalized Ricci soliton; and, under the compactness assumption, \cite{GFStr} guaranties that this soliton is actually a gradient one.

\begin{thm} \label{thm:5 intro}
Let $(M,\phi,\xi,g,\nabla)$ be a compact $5$-dimensional $\nabla$-Hermite–Einstein manifold, and let $f$ be the unique normalized smooth function such that $V = \theta^{\sharp} - \operatorname{grad} f$ is parallel. Then:

\begin{enumerate}
\item If $V \neq 0$, then $\nabla$ is flat (Theorem~\ref{thm:flat}) and the universal cover of $M$ is isometric to $\mathbb{R} \times (\mathbb{R} \times \mathrm{SU}(2))$.

\item If $V = 0$, we distinguish according to the constant $c$ determined by the $\nabla$-Einstein condition:
\begin{enumerate}
    \item When $c = 0$, $M$ is a mapping torus over a compact K\"ahler Ricci-flat $4$-manifold.
    \item When $c \neq 0$:
    \begin{enumerate}
        \item If $f$ is constant (equivalently, $\tilde s^T$ is constant), then $\nabla$ is flat and the universal cover is isomorphic to $\mathrm{SU}(2) \times \mathbb{C}$; the lifted almost contact metric structure induces the standard left-invariant Sasaki structure on the $\mathrm{SU}(2)$ factor.
        \item If $f$ is non‑constant, then $\nabla$ is non‑flat and $(M,\phi,\xi,g,\nabla)$ is locally an $S^1$-bundle over a $4$-dimensional K\"ahler manifold with strictly positive non--constant scalar curvature $\tilde s^T$ satisfying
        \[
        \Box \tilde s^T  = \frac{(\tilde s^T)^2}{2} - |\widetilde{\mathrm{Ric}}|^2.
        \]
    \end{enumerate}
\end{enumerate}
\end{enumerate}
\end{thm}

Notably, the emerging transverse $4$-dimensional geometry coincides precisely with that of six-dimensional Bismut-Hermite--Einstein manifolds \cite{ABLS}. Observe, moreover, that while in \cite{ABLS} the vanishing of $V$ enforces  K\"ahlerianity, in the present setting it leads to a significantly richer geometry.\par
\medskip
In Section~6 we turn to geometric flows of almost contact metric structures. To the best of our knowledge, this subject has received comparatively little attention. We begin by developing a general framework for such flows, viewing an almost contact metric structure as a $\mathrm{U}(n)'$-structure and decomposing its infinitesimal deformation space into irreducible $\mathrm{U}(n)'$-modules. Here $U(n)'$ is viewed as $\{1\}\times U(n)\subset \SO(2n+1)$,
acting trivially on the Reeb direction. This allows us to determine explicitly which components of a deformation act on the metric, the Reeb vector field, the contact form, and the endomorphism $\phi$, and hence to construct flows preserving selected parts of the structure.

Our main purpose is to introduce a natural evolution equation for Sasaki with torsion structures. The construction is motivated by the pluriclosed flow in Hermitian geometry and by its gauge equivalence with the generalized Ricci flow. Two issues arise in the odd-dimensional setting. First, although a natural geometric flow preserves the symmetry generated by an initially Killing Reeb vector field, it need not preserve its length. Second, an arbitrary flow of almost contact metric structures does not preserve normality. We overcome the first difficulty by considering a suitable normalization of the Ricci flow, which preserves the unit length of the fixed Reeb vector field. To address the second one, we keep fixed the complex structure induced by $\phi$ on the quotient bundle $TM/\langle\xi\rangle$ and evolve $\eta$ within the class of normal almost contact structures.

A key feature of this evolution is that it preserves the closedness of the
torsion, namely the condition $dH=0$. When the initial Sasaki with torsion
structure is strong, we prove that the induced evolution is gauge-equivalent
to the generalized Ricci flow. Moreover, after passing to the trivial product
$M\times S^1$, the corresponding Hermitian evolution is, up to a natural
gauge transformation along the $S^1$-factor, precisely the pluriclosed flow.
These results show that the parallel between Sasaki with torsion geometry
and Hermitian geometry persists at the level of geometric flows. 

More precisely, setting
\[
c=\Lambda_Fd\eta,
\]
we consider the evolution equations
\[
\begin{aligned}
\partial_t\eta
&=
-d^{\phi_0}c,\\
\partial_tg^T
&=
-2\Ric(g^T)
+2d\eta^2
+\frac12(d^\phi F)^2
-\mathcal L_{\theta^\sharp}g^T,
\end{aligned}
\]
where $\theta$ is the Lee form of the transverse Hermitian structure. The first equation preserves the normality condition, whereas the second is tangent to the space of transverse Hermitian metrics. The metric evolves by the Ricci flow to highest order, while the equation for $\eta$ has the transverse Hodge Laplacian as its principal part. After an appropriate DeTurck gauge fixing, the resulting system is transversely parabolic.

The main results concerning this flow can be summarized as follows.

\begin{thm}\label{thm:flow intro}
Let $(M^{2n+1},\phi_0,\eta_0,\xi,g_0)$ be a compact Sasaki with torsion manifold. Then the following statements hold.

\begin{enumerate}
\item The evolution equations above admit a unique solution
\[
(\phi_t,\eta_t,\xi,g_t)
\]
for $t\in[0,\varepsilon)$, for some $\varepsilon>0$, which remains Sasaki with torsion.

\item If the initial structure is strong, then the strong condition is preserved along the flow. 

\item If the initial structure is strong, the induced pair
\[
g_t=\eta_t\otimes\eta_t+g_t^T,
\qquad
H_t=\eta_t\wedge d\eta_t+d^{\phi_t}F_t,
\]
satisfies the gauge-fixed generalized Ricci flow
\[
\begin{aligned}
\partial_tg
&=
-2\Ric(g)+\frac12H^2-\mathcal L_{\theta^\sharp}g,\\
\partial_tH
&=
\Delta H-\mathcal L_{\theta^\sharp}H.
\end{aligned}
\]
\end{enumerate}
\end{thm}

The short-time existence and uniqueness in the first statement correspond to Theorem~\ref{thm:ST-flow}. The preservation of the strong condition is Proposition~\ref{prop:SST-preserved}; its proof follows by showing that $\Psi=dH$ satisfies a homogeneous linear transversely parabolic equation. The final statement identifies the flow, on the strong locus, as natural odd-dimensional counterpart of the pluriclosed flow \cite{StreetsTian, ST2}. This analogy is particularly striking after the following observation (see Theorem \ref{thm:SWT-pluriclosed}): on the strong locus, after passing to the product $M\times S^1$ and applying a natural gauge transformation along the circle factor, the induced Hermitian evolution coincides with the pluriclosed flow with respect to the initial complex structure. The following theorem hence holds: 
\begin{thm}
Let $\Phi_t=(\varphi_t,\xi,\eta_t,g_t)$ be a solution of the
Sasaki-with-torsion flow starting from a compact strong
Sasaki-with-torsion manifold $(M, \Phi_0)$, and let
$(\widehat J_t,\widehat\omega_t)$ be the induced Hermitian structure
on $M\times S^1_{\tau}$. Set $
c_t:=\Lambda_{F_t}d\eta_t,$ 
and let $\chi_t$ be the family of diffeomorphisms generated by
$c_t\partial_\tau$, with $\chi_0=\operatorname{Id}$. Then
\[
\chi_t^*\widehat J_t=\widehat J_0
\]
and $\chi_t^*\widehat\omega_t$ is the solution of the pluriclosed flow
on $(M\times S^1,\widehat J_0)$, normalized by
$
\partial_t\omega
=
-2\left(\rho^B_\omega\right)^{1,1},
$
with initial datum $\widehat\omega_0$. In particular, if
$\omega_t^{\mathrm{PC}}$ denotes this solution, then
$
\chi_t^*\widehat\omega_t=\omega_t^{\mathrm{PC}}.
$
\end{thm}

\medskip 
We also characterize its stationary points within the strong class. A strong Sasaki with torsion structure is stationary if and only if
\[
dc=0,
\qquad
(\rho^B)^{1,1}=c\,d\eta.
\]
In particular, these static points are precisely $\nabla$-Einstein structures (Proposition \ref{prop:stationary_sst}). Thus, the evolution introduced in Section~6 is naturally adapted to the odd-dimensional analogue of the Bismut--Hermite--Einstein condition studied in Section~5.
\noindent 
\newline 

\smallskip

{\bf Acknowledgments:} The authors would like to thank Jeff Streets,  Diego Conti  and Vestislav Apostolov for very helpful comments and discussions.
 Anna Fino is partially supported by Project PRIN 2022 \lq \lq Geometry and Holomorphic Dynamics”, by GNSAGA (Indam) and by a grant from the Simons Foundation (\#944448).  Gueo Grantcharov is partially supported by a grant from the Simons Foundation (\#853269).

\section{Preliminaries}\label{Sec:pre}
We begin this section by briefly recalling the main definitions and properties that will be needed in the sequel. 
For a more exhaustive treatment of the theory of almost contact metric manifolds, we refer the reader to \cite{Bl,BG}.

\medskip

An \emph{almost contact manifold} is a smooth manifold \(M\) of dimension \(2n+1\) endowed with a triple 
\((\phi, \eta, \xi)\), where \(\phi\) is a \,(1,1)-tensor field, \(\eta\) is a \(1\)-form, and \(\xi\) is a vector field satisfying
\begin{equation} \label{eqn:phisquare}
    \phi^2 = -\mathrm{Id} + \eta \otimes \xi, 
    \qquad 
    \eta(\xi) = 1.
\end{equation}
The vector field \(\xi\) is usually called the \emph{Reeb vector field}. \par
Given an almost contact manifold \((M, \phi, \eta, \xi)\), the tangent bundle decomposes as
\begin{equation} \label{eqn:splitting}
    TM = \mathcal{F}_\xi \oplus \ker(\eta),
\end{equation}
where \(\mathcal{F}_\xi\) is the line distribution spanned by \(\xi\). This follows by the fact that any vector field $X$ can be uniquely written as $X=\eta(X) \xi + (X -\eta(X) \xi)$ and $\mathcal{F}_\xi \cap\ker(\eta)=\{0\}$. The distribution \(\ker(\eta)\) is usually called the \emph{horizontal distribution}. \par

It follows from \eqref{eqn:phisquare}  that the following additional relations hold:
\begin{equation} \label{eqn:phieta}
    \eta(\phi X) = 0, 
    \qquad 
    \phi(\xi) = 0, 
\end{equation}
An almost contact structure \( (\phi, \eta, \xi) \) is said to admit a \emph{compatible Riemannian metric} \(g\) if
\begin{equation} \label{eqn:metriccomp}
    g(\phi X, \phi Y) = g(X, Y) - \eta(X)\eta(Y),
\end{equation}
for all vector fields \(X, Y \in \mathfrak{X}(M)\).
A quadruple \((\phi, \eta, \xi, g)\), where \(g\) is compatible, is called an \emph{almost contact metric structure}.

It follows from \eqref{eqn:metriccomp} that:
\begin{equation} \label{eqn:phieta2}
    \eta = \xi^\flat, \quad g(\xi,\xi)=1
\end{equation}
where \(\xi^\flat\) denotes the metric dual of \(\xi\). In this case the decomposition obtained in \eqref{eqn:splitting} is an \emph{orthogonal} decomposition. Orthogonality follows from \(\eta = \xi^\flat\).

Since \(\phi(\xi) = 0\), the tensor field \(\phi\) preserves the splitting: it vanishes on \(\mathcal{F}_\xi\) and restricts to an almost complex structure on \(\ker(\eta)\).
If we denote by \(g^T\) the restriction of \(g\) to \(\ker(\eta)\), then 
\((\phi|_{\ker(\eta)}, g^T)\) defines a horizontal almost Hermitian structure. 
In particular, \(\phi|_{\ker(\eta)}\) induces a decomposition of horizontal complex forms into \((p,q)\)-types. \par
We point out that in the following we will use the following convention: for any $\alpha$ horizontal $k$-form and for any $X_1,\dots,X_k$ vector fields:
\[\phi \alpha(X_1,\dots,X_k):=(-1)^k \alpha(\phi X_1,\dots,\phi X_k).\]

\begin{rmk}
Locally, one may choose an orthonormal frame, called a \(\phi\)-basis,
\[
\{\xi, x_1, \phi x_1, \dots, x_n, \phi x_n\},
\]
where \(x_i,  \phi x_i\) form an orthonormal basis of horizontal vector fields.  
Moreover, the complex vector fields \(\{x_i - i\,\phi(x_i)\}\) form a unitary basis of horizontal \((1,0)\)-vectors.
\end{rmk}
The local geometry of an almost contact metric structure can be described as follows: 
\begin{prop}\label{prop: local model}
    There exists local coordinates $\{s,u_1,...,u_{2n}\}$ on $M$ so that $\xi=\partial_s$ and using Einstein summation convention the following holds:
    \begin{gather*}
        \eta = ds+ f_idu_i\\
        g=\eta\otimes \eta+g^T_{ij}du_i\otimes du_j,\\
        \varphi=(\varphi^{j}_i\partial_{u_i}-f_k\varphi^{k}_i\partial_s)\otimes du_i,
    \end{gather*}
where $f_i,g_{ij}^T, \varphi_i^j$ are local functions on $M$ and $\varphi_k^j\varphi_i^k=-\delta_i^j$. The horizontal distribution is then spanned by $\{\partial_{u_i}-f_i\partial_s\}_{i=1}^{2n}$ and the compatibility condition means that $g^T_{ij}$ is almost Hermitian with respect to $\varphi^j_i$. 
\end{prop}
The above proposition can be specialized to various special classes of almost contact metric structures, see for instance \cite{Pawel2000} for the Sasaki case.

Associated with an almost contact metric structure is the \(2\)-form
\begin{equation}\label{def of F}
F(X,Y) = g(X, \phi Y),    
\end{equation}
called the \emph{fundamental form}.
Since \(\phi\xi = 0\), the form \(F\) is non-zero only on horizontal vectors.
Using \eqref{eqn:metriccomp}, one checks immediately that \(F(X,Y) = F(\phi X, \phi Y)\).
Thus \(F\) plays the role of the Kähler form of the horizontal almost Hermitian structure.
It is clear that \(\xi \wedge F^n\) defines a volume form on \(M\) up to a constant factor.

\medskip

\begin{defn}
An almost contact structure is said to be \emph{normal} if
\[
N_\phi + d\eta \otimes \xi = 0,
\]
where 
\[
N_\phi(X,Y) = \phi^2[X,Y] + [\phi X, \phi Y] - \phi[\phi X, Y] - \phi[X, \phi Y]
\]
is the Nijenhuis tensor of \(\phi\).
\end{defn}

Equivalently, an almost contact metric structure is normal if the almost complex structure \(J\) on \(M \times S^1\) defined by
\begin{equation} \label{eqn:cpxstructure}
    J(X, f\,\partial_s) = (\phi X - f \xi,\ \eta(X)\,\partial_s)
\end{equation}
is integrable, where \(X \in \mathfrak{X}(M)\), \(f \in C^\infty(M)\), and \(\partial_s\) is the coordinate vector field on \(S^1\).

\begin{lem}[\cite{Bl,BG}] \label{lem:blair}
For a normal almost contact metric structure, the following hold:
\begin{align}
&\mathcal{L}_\xi \eta = \iota_\xi d\eta = 0, \label{eqn:xieta}\\
&d\eta(\phi X, Y) + d\eta(X, \phi Y) = 0, \qquad X,Y \in \mathfrak{X}(M), \label{eqn:dn11} \\
&\mathcal{L}_\xi \phi (X) = [\xi, \phi X] - \phi[\xi,X] = 0, \qquad X \in \mathfrak{X}(M). \label{eqn:xiphi}
\end{align}
\end{lem}
 
\medskip

\begin{rmk}
In what follows, we adopt the following notational convention:
\begin{enumerate}
    \item Capital letters \(X, Y, Z, \dots\) denote arbitrary vector fields on \(M\).
    \item Lowercase letters \(x, y, z, \dots\) denote horizontal vector fields.
\end{enumerate}
\end{rmk}

The Friedrich-Ivanov connection is defined as follows.

\begin{thm}[\cite{FI,FI2}] \label{thm:nabla}
Let \((M,\phi,\xi,\eta,g)\) be a normal almost contact metric manifold.  
There exists a unique metric connection \(\nabla\) with totally skew-symmetric torsion satisfying 
\begin{equation} \label{eqn:nabla}
    \nabla \phi = 0, \qquad \nabla \eta = 0,
\end{equation}
if and only if the Reeb vector field \(\xi\) is Killing.  
If it exists, \(\nabla\) is given by
\begin{equation} \label{eqn:nabla2}
    g(\nabla_X Y, Z) 
    = g(\nabla^g_X Y, Z) 
      + \frac{1}{2} \bigl(\eta \wedge d\eta + d^\phi F \bigr)(X,Y,Z),
\end{equation}
where \(d^\phi F(X,Y,Z) = -dF(\phi X, \phi Y, \phi Z)\).
\end{thm}
\begin{defn} [\cite{CoMa}]
 Let \((M,\phi,\xi,\eta, g)\) be a normal almost contact metric manifold with Killing Reeb vector field. The quintuple \((M,\phi,\xi, \eta, g, \nabla)\) will be called a \emph{Sasaki with torsion structure}.    
\end{defn}
We note that the Friedrich--Ivanov connection is defined under the weaker assumption that the Nijenhuis tensor does not vanish but is totally skew-symmetric, that is,
\[
N_\phi(X,Y,Z) = g\bigl(N_\phi(X,Y),Z\bigr) \in \Omega^3(M).
\]
However, since our primary interest lies in the five-dimensional case, we may assume the almost metric contact structure to be normal. Indeed, in dimension five, the condition that the Nijenhuis tensor be totally skew-symmetric is equivalent to its vanishing \cite{CM}.\par
Since \(\nabla g = 0\), equations \eqref{eqn:nabla} imply immediately that \(\nabla \xi = 0\).
There is an evident analogy between \(\nabla\) and the Bismut connection of a Hermitian manifold; in particular, $
\mathrm{Hol}(\nabla) \subseteq \mathrm{U}(n).$

Moreover, \(\nabla\) coincides with the Levi--Civita connection \(\nabla^g\) if and only if both \(d\eta\) and \(dF\) vanish.
Equivalently, this occurs precisely when \(M\) is co--Kähler. In the compact case, $M$ has the structure of a mapping torus of a Kähler manifold via a Kähler isometry \cite{Li}.
In this case, the almost contact metric structure is parallel with respect to \(\nabla^g\), exactly as in the Kähler setting.

\medskip

Examples of normal almost contact metric manifolds with Killing Reeb vector field include Sasaki and quasi-Sasaki manifolds, both having Kähler transverse geometry.

More generally, any normal almost contact metric manifold with Killing Reeb vector field is \emph{locally} an \(S^1\)-bundle with curvature \(d\eta\) over a Hermitian manifold \((U, \phi|_{\ker(\eta)}, g^T)\) with fundamental form \(F\).

Conversely, any \(S^1\)-bundle (orbibundle) over a Hermitian manifold (orbifold) whose curvature form is of type \((1,1)\) with respect to the complex structure defines a normal almost contact metric manifold with Killing Reeb vector field \cite{CoMa}

\medskip

For horizontal vector fields \(x,y,z\), equation \eqref{eqn:nabla2} gives
\[
g(\nabla_x y, z) = g(\nabla^g_x y, z) + \tfrac12 d^\phi F(x,y,z) = g(\nabla^B_x y, z),
\]
where \(\nabla^B\) is the Bismut connection of the transverse Hermitian structure.
Indeed, for horizontal vectors, \(d^\phi F\) agrees with the torsion of the transverse Bismut connection. 
\begin{rmk}\label{rmk:rmk2}
Let $x$ be a horizontal vector field.  
Since $\nabla$ preserves both the Reeb foliation $\mathcal{F}_\xi$ and the horizontal distribution $\ker(\eta)$, we have $\nabla x \in \ker(\eta)$.  
In particular, for any horizontal vector field $y$, the only nonzero components of $\nabla_y x$ lie in the horizontal distribution.  

  Therefore, for any horizontal $z$,
\[
g(\nabla_y x, z) = g(\nabla^B_y x, z),
\]
and hence
\[
\nabla_y x = \nabla^B_y x.
\]
In other words, the connection $\nabla$ agrees with the transverse Bismut connection $\nabla^B$ on horizontal vector fields. The curvature of $\nabla^B$ will be denoted by $R^B$ throughout the paper.
\end{rmk}

We conclude this section with the following two remarks:
\begin{rmk} \label{rmk:rmk}
We record here some immediate consequences that hold on any normal almost contact metric manifold with Killing Reeb vector field:
\begin{enumerate}
    \item \(\mathcal{L}_\xi F = \iota_\xi dF = 0\), as a consequence of \(\mathcal{L}_\xi g = 0\) and \eqref{eqn:xiphi}.
    \item \(\iota_\xi d^\phi F = 0\), by equation \eqref{eqn:phieta}.
    \item \(\mathcal{L}_\xi d^\phi F = \iota_\xi d d^\phi F = 0\), using \(\mathcal{L}_\xi F = 0\) (hence \(\mathcal{L}_\xi dF = 0\)) and equation \eqref{eqn:xiphi}.
    \item \(\nabla_X d\eta(\xi, \cdot) = 0\), using equation \eqref{eqn:xieta} and \(\nabla \xi = 0\).
    \item \(\nabla_X d^\phi F(\xi, \cdot, \cdot) = 0\), using item~2 and \(\nabla \xi = 0\).
    \item $[\xi, \mathcal{F}^\perp]\subset \mathcal{F}^\perp_\xi$, using that $\mathcal{L}_{\xi} g(\xi,x)=0$ for any $x \in \mathcal{F}^\perp_\xi$.
    \item $\mathcal{L}_\xi g^T=0$ and $\mathcal{L}_\xi \phi_{| \mathcal{F}^\perp_\xi}=0$, using item 6. 
\end{enumerate}
\end{rmk}
In the following an almost contact metric manifold will be denoted simply by $(\phi, \xi,g)$. In fact, the $1$-form $\eta$ is recovered from the relation $\eta=\xi^\flat$. \par
\medskip
We include here the proof of the following proposition, which is well known in the Sasaki case by \cite{BGM,Smoczyk2010}. 
\begin{prop} \label{prop: normality variation}
Let $(\phi_0, \xi_0,\eta_0,g_0^T)$ be a normal almost contact structure and let $f$ be a basic function, that is $df(\xi_0)=0$. Then the structure $(\phi, \xi_0,\eta,g_0^T)$ where
\[
\eta:=\eta_0+d^{\phi_{0}}f, \, \, \, \, \phi:=\phi_0-\xi_0 \otimes df,  
\]
is again normal almost contact. 
\end{prop}
\begin{proof}
The fact that $\eta(\xi_0)=1$ immediately follows since $\eta_0(\xi_0)=1$ and $f$ is basic. Now, we check that $\phi^2=-\mathrm{Id}+\eta \otimes \xi_0$. 
\begin{align*}
\phi^2(X)&=(\phi_0-\xi_0 \otimes df) (\phi_0(X)-df(X) \xi_0 )\\
&=\phi_0^2(X) - df(\phi_0 X)\xi_0\\
&= -X + (\eta_0+d^{\phi_{0}}f)(X) \xi_0.  
\end{align*}
where we used that $\phi_0(\xi_0)=0$ and again that $f$ is basic. To conclude the result we need to check the normality condition
\[
\phi^2[X,Y] + [\phi X, \phi Y] - \phi[\phi X, Y] - \phi[X, \phi Y]+d\eta \otimes \xi_0=0.
\]
We analyze each term individually: 
\begin{align*}
\phi^2[X,Y]&=-[X,Y]+\eta_0([X,Y]) \xi_0 + d^{\phi_{0}}f([X,Y]) \xi_0; \\
 [\phi X, \phi Y]&= [\phi_0 X, \phi_0 Y]-(\phi_0 X (df(Y))) \xi_0  +(\phi_0 Y (df(X))) \xi_0+df(X) \, df([\xi_0, Y]) \xi_0 \\&  \ \ \ \  -df(Y) \, df([\xi_0, X]) \xi_0-df(Y) \, [\phi_0 X, \xi_0]-df (X) \, [\xi_0, \phi_0 Y];\\
 -\phi [\phi X, Y]&= -\phi_0([\phi_0 X,Y])+df([\phi_0 X,Y]) \xi_0+df(X) \, \phi_0[\xi_0, Y]-df(X) df([\xi_0, Y]) \xi_0; \\
 -\phi [X, \phi Y ]&= \phi_0([\phi_0 Y,X])-df([\phi_0 Y,X]) \xi_0-df(Y) \, \phi_0[\xi_0, X]+df(Y) df([\xi_0, X]) \xi_0.
\end{align*}
The terms $-df(X)[\xi_0, \phi_0 Y]$ in the second expression cancel with $df(X) \,  \phi_0 [\xi_0, Y]$ in the third one since $\mathcal{L}_{\xi} \phi_0=0$. The same argument apply to $-df(Y)[\phi_0 X, \xi_0]$ in the second expression and $-df(Y) \, \phi_0 [\xi_0, X]$ in the fourth. Furthermore, $df(X) df([\xi_0,Y]) \xi_0$ in the second expression cancels with the same terms of opposite sign in the third, and the same argument applies also to $df(Y) df([\xi_0,X]) \xi_0$. \par
Collecting all the remaining terms, we get
\begin{align*}
&N_{\phi_{0}}(X,Y)+ ( d^{\phi_{0}}f([X,Y]) -(\phi_0 X (df(Y)))  +(\phi_0 Y (df(X)))+df([\phi_0 X,Y])-df([\phi_0 Y,X]) ) \xi_0=\\
& N_{\phi_{0}}(X,Y)-Y(df(\phi_0 X))+X(df(\phi_0 Y))+\phi_0 df([X,Y])=-\left(d\eta_0(X,Y) +d d^{\phi_{0}} f (X,Y)\right) \xi_0,
\end{align*}
where the second line follows by explicitly writing the brackets and the last line follows since the initial almost contact structure is normal. 
\end{proof}

Throughout this paper, all manifolds are assumed to be connected, unless explicitly stated otherwise.

\section{Co-K\"ahler-like manifolds}
Throughout this section, \((M, \phi, \xi, g, \nabla)\) denotes a Sasaki with torsion manifold.

\begin{defn}
Let \((M, \phi, \xi, g, \nabla)\) be a Sasaki with torsion manifold.  
We say that \((M, \phi, \xi, g)\) is \emph{co--K\"ahler-like} if the connection \(\nabla\) satisfies the first Bianchi identity.
\end{defn}

The notion of co--K\"ahler-likeness arises as a natural analogue of a concept introduced by Angella, Otal, Ugarte and Villacampa in the Hermitian setting \cite{AOUV}.  
There, a Hermitian manifold is said to be \emph{K\"ahler-like} when the Bismut connection satisfies the first Bianchi identity.  
A conjecture proposed in that work was later proved by Zhao and Zheng, who obtained the following characterization.

\begin{thm}[\cite{ZZ}] \label{thm:ZZ}
Let \((M, J, g)\) be a Hermitian manifold.  
The Bismut connection \(\nabla^B\) satisfies the first Bianchi identity if and only if its torsion \(3\)-form is both parallel and closed.
\end{thm}

In what follows, we establish an analogous result for the Friedrich-Ivanov connection on almost contact metric manifolds.

\begin{thm} \label{thm:CKL}
Let \((M, \phi, \xi, g, \nabla)\) be a Sasaki with torsion manifold. Then $(\phi, \xi, g)$ is co--K\"ahler-like if and only if the torsion \(3\)-form
\[
H = \eta \wedge d\eta + d^{\phi} F
\]
is both closed and parallel.
\end{thm}

Before proving this result, we recall two standard facts about metric connections with skew-symmetric torsion.  
These results are classical and hold for any such connection; for details we refer to \cite{AFF, IS} and references therein.

\medskip

Since \(\nabla\) is not torsion-free, the first Bianchi identity does not automatically hold.  
Indeed,
\begin{equation} \label{eqn:1bianchi}
\mathfrak{S}_{X,Y,Z} R(X,Y,Z,W)
    = dH(X,Y,Z,W)
    + (\nabla_W H)(X,Y,Z)
    - \sigma_H(X,Y,Z,W),
\end{equation}
where \(R=[\nabla,\nabla]-\nabla_{[\cdot,\cdot]}\) is the curvature tensor of \(\nabla\), and
\begin{equation} \label{eqn:sigma}
  \sigma_H(X,Y,Z,W)
    = \mathfrak{S}_{X,Y,Z}
      g\big(H(X,Y), H(Z,W)\big).  
\end{equation}

The next fundamental result was established in \cite{IS}.

\begin{thm}[\cite{IS}] \label{thm:IS}
Let \((\bar{M}, \bar{g})\) be a Riemannian manifold endowed with a metric connection \(\bar{\nabla}\) whose torsion is the skew-symmetric \(3\)-form \(\bar{H}\).  
Then the curvature of \(\bar{\nabla}\) satisfies the Riemannian first Bianchi identity if and only if
\begin{equation} \label{eqn:IS}
    d\bar{H} = -2\, \bar{\nabla}\bar{H} = \tfrac{2}{3}\, \sigma_{\bar{H}}.
\end{equation}
\end{thm}

From this theorem, the implication in Theorem~\ref{thm:CKL} from right to left is immediate, in fact, if the torsion is both closed and parallel, then also $\sigma_{\bar{H}}$ vanishes (see, for instance \cite[ Equation (2.4)]{IS}). \par  
For later use, we also recall the following lemma, which clarifies the relation between the symmetries of \(\bar{\nabla}\bar{H}\) and those of the curvature tensor.

\medskip

We introduce the notation
\[
\begin{split}
\mathrm{Sym}\!\left(\Lambda^2 T^*M\right) := \Big\{
T \in (T^*M)^{\otimes 4} \ \Big| \ 
&T(X,Y,Z,W) = -T(Y,X,Z,W) = -T(X,Y,W,Z),\\
&T(X,Y,Z,W) = T(Z,W,X,Y)
\Big\}.
\end{split}
\]

\begin{lem}[\cite{Iv}] \label{lem:symm2}
Let \((\bar{M}, \bar{g})\) be a Riemannian manifold endowed with a metric connection \(\bar{\nabla}\) with skew-symmetric torsion \(\bar{H}\).  
Then the following are equivalent:
\[
(\bar{\nabla}_X \bar{H})(Y,Z,W) = -(\bar{\nabla}_Y \bar{H})(X,Z,W)
    \quad\Longleftrightarrow\quad
    \bar{R}(X,Y,Z,W) = \bar{R}(Z,W,X,Y).
\]
In particular, if \(\bar{\nabla}\) satisfies the first Bianchi identity, then \(\bar{R} \in \mathrm{Sym}\!\left(\Lambda^2 T^*M\right)\), and \(\bar{\nabla}\bar{H}\) may be regarded as a \(4\)-form.
\end{lem}

\medskip

We now turn to the proof of the converse implication in Theorem~\ref{thm:CKL}.

\begin{proof}[Proof of Theorem~\ref{thm:CKL}]
We organize the proof into six steps.

\medskip
\noindent
\textbf{Step 1: \(\nabla d\eta = 0\).}  
\\
Recall that
\[
H = \eta \wedge d\eta + d^\phi F, \qquad \text{and} \qquad \nabla \eta = 0.
\]
For arbitrary vector fields \(X,Y,Z\),
\[
\nabla_X H(Y,Z,\xi) = \big(\eta \wedge \nabla_X d\eta + \nabla_X d^\phi F\big)(Y,Z,\xi) = \nabla_X d\eta(Y,Z),
\]
since, by Remark~\ref{rmk:rmk} (items 4 and 5),
\[
\nabla_X d\eta(\xi,\cdot) = 0, \quad 
\nabla_X d^\phi F(\xi,\cdot,\cdot) = 0.
\]
On the other hand, Theorem~\ref{thm:IS}, Remark~\ref{rmk:rmk} (item 3), and Equation \eqref{eqn:xieta} give
\[
\nabla_X d\eta(Y,Z) = \nabla_X H(Y,Z,\xi) = -\tfrac{1}{2} dH(X,Y,Z,\xi) = -\tfrac{1}{2} \big( d\eta \wedge d\eta + d d^\phi F \big)(X,Y,Z,\xi) = 0.
\]
Hence,
\[
\nabla d\eta = 0.
\]

\medskip
\noindent
\textbf{Step 2: Reduction to \(\nabla^B d^\phi F = 0\).} \\ 
By Step~1,
\[
\nabla H = \nabla d^\phi F \in \Omega^4(M),
\]
and the Reeb vector field \(\xi\) lies in the kernel of \(\nabla d^\phi F\) (Remark~\ref{rmk:rmk} item 5).  
Thus \(\nabla d^\phi F\) defines a horizontal 4-form. \par
By Remark \ref{rmk:rmk2}, for horizontal vector fields \(x,y,z,w\), we have
\begin{equation}\label{eqn:nablaB}
(\nabla_x d^\phi F)(y,z,w) = (\nabla^B_x d^\phi F)(y,z,w),
\end{equation}
where \(\nabla^B\) is the Bismut connection of the transverse Hermitian structure.  
Proving \(\nabla^B d^\phi F = 0\) suffices to conclude \(\nabla H = 0\). Then Theorem~\ref{thm:IS} implies \(dH = 0\).

\medskip
\noindent
\textbf{Step 3: Symmetries of \(R^B\).} \\
By Step~2, \(\nabla^B d^\phi F\) is a horizontal 4-form, so Lemma~\ref{lem:symm2} implies
\[
R^B(x,y,z,w) = R^B(z,w,x,y)
\]
for all horizontal vector fields \(x,y,z,w\).

\medskip
\noindent
\textbf{Step 4: Transverse first Bianchi identity.}  \\
For horizontal vectors \(x,y,z,w\), equation \eqref{eqn:1bianchi} becomes
\begin{equation}\label{eqn:1bianchi*}
0 = (dd^\phi F + d\eta \wedge d\eta)(x,y,z,w) + \nabla^B_w d^\phi F(x,y,z) - \sigma_{d^\phi F}(x,y,z,w) - \tfrac{1}{2} d\eta \wedge d\eta(x,y,z,w),
\end{equation}
where we used equation \eqref{eqn:nablaB} and
\[
\sigma_H(x,y,z,w) = \sigma_{d^\phi F}(x,y,z,w) + \frac{1}{2} d\eta \wedge d\eta(x,y,z,w).
\]
Rewriting \eqref{eqn:1bianchi*} using Equation \eqref{eqn:1bianchi} for the transverse Bismut connection,
\[
0 = \mathfrak{S}_{x,y,z} R^B(x,y,z,w) + \tfrac{1}{2} d\eta \wedge d\eta(x,y,z,w).
\]

\medskip
\noindent
\textbf{Step 5: characterization of \(\nabla^B d^\phi F = 0\).}  \\
We use the curvature characterization of Hermitian structures with parallel Bismut torsion \cite{ZZ2}.  
For horizontal \((1,0)\) vectors \(x,y,z,w\), a Hermitian structure has parallel Bismut torsion if and only if
\begin{align}
& R^B(x,y,z,\overline{w}) = 0, \label{eqn:1}\\
& R^B(x,\overline{y},z,\overline{w}) = R^B(z,\overline{w},x,\overline{y}), \label{eqn:2}\\
& \nabla^B \mathrm{Ric}(Q) = 0, \label{eqn:3}\\
& \mathrm{Ric}(Q)_{\chi \overline{y}} = 0, \label{eqn:4}
\end{align}
where \(2 \chi = \theta^\sharp - i \phi(\theta^\sharp)\), with $\theta^\sharp$ being the Lee field, and
\[
\mathrm{Ric}(Q)_{x\overline{y}} = \sum_{i=1}^n \Big(R^B(x,\overline{y},e_i,\overline{e_i}) - R^B(e_i,\overline{y},x,\overline{e_i})\Big),
\]
for a unitary frame \(\{e_i\}\) of horizontal \((1,0)\) vectors \footnote{We point out that we are considering the extension of $g^T$ over $\C$ to be the bi-linear one.}.

By Step~3, \eqref{eqn:1} and \eqref{eqn:2} hold. Hence we need to prove \eqref{eqn:3} and \eqref{eqn:4}. Observe first that
\begin{equation*}
\begin{split}
\mathrm{Ric}(Q)_{x\overline{y}}
&= \sum_{i=1}^n \Big( R^B(x,\overline{y},e_i,\overline{e_i}) - R^B(e_i,\overline{y},x,\overline{e_i}) \Big)\\[2mm]
&= \sum_{i=1}^n \Big( R^B(x,\overline{y},e_i,\overline{e_i}) + R^B(\overline{y},x,e_i,\overline{e_i}) + R^B(x,e_i,\overline{y},\overline{e_i}) \Big) 
   + \tfrac{1}{2} \sum_{i=1}^n (d\eta \wedge d\eta)(e_i,\overline{y},x,\overline{e_i})\\[1mm]
&= \tfrac{1}{2} \sum_{i=1}^n (d\eta \wedge d\eta)(e_i,\overline{y},x,\overline{e_i})\\
&= -\tfrac{1}{2} \sum_{i=1}^n (d\eta \wedge d\eta)(x,\overline{y},e_i,\overline{e_i}),
\end{split}
\end{equation*}
where the second line uses the transverse first Bianchi identity (Step~3), and the term \(R^B(x,e_i,\overline{y},\overline{e_i})\) vanishes since equation \eqref{eqn:1} holds.

\medskip
Next we verify \eqref{eqn:3}. By Remark~\ref{rmk:rmk2}, for any horizontal \(x,y,z\) we have
\[
0 = \nabla_x d\eta(y,z) = \nabla^B_x d\eta(y,z),
\]
and therefore
\[
\nabla^B_x(d\eta \wedge d\eta) = 0 \quad \text{for all horizontal } x.
\]
Fix a point \(p \in M\) and choose a local unitary frame \(\{e_i\}\) of horizontal \((1,0)\)-vectors that is \(\nabla\)-parallel at \(p\). By Remark~\ref{rmk:rmk2}, we also have \(\nabla^B_x e_i = 0\) at \(p\) for any horizontal \(x\) (similarly for \(\overline{e_i}\)). Therefore, for any horizontal vector field \(z\), we have that at \(p\),
\[
\begin{split}
\nabla^B_z \mathrm{Ric}(Q)(x,\overline{y})
&= -\tfrac{1}{2} \sum_{i=1}^n \nabla^B_z (d\eta \wedge d\eta)(x,\overline{y},e_i,\overline{e_i})= 0.
\end{split}
\]
Hence \(\nabla^B \mathrm{Ric}(Q) = 0\), proving \eqref{eqn:3}.

\medskip
Finally, we check \eqref{eqn:4}. Using the above expression,
\[
\mathrm{Ric}(Q)_{\chi \overline{y}} = -\tfrac{1}{2} \sum_{i=1}^n (d\eta \wedge d\eta)(\chi, \overline{y}, e_i, \overline{e_i}).
\]
Therefore, \(\mathrm{Ric}(Q)_{\chi \overline{y}} = 0\) will follow if \(\chi \in \ker(d\eta)\). We devote the final step to proving this.

\medskip
\noindent
\textbf{Step 6: \(\chi \in \ker(d\eta)\).}  \\
Let \(\{\xi, x_i, \phi x_i\}\) be a \(\phi\)-basis. Without loss of generality, assume that at a point \(p\), \(\nabla x_i = 0\). Define the 1-form (\cite{FI})
\[
\omega^\nabla(X) := -\sum_{i=1}^n H(X,x_i,\phi x_i) = c \eta(X) + \phi \theta(X),
\]
where \(\theta\) is the Lee form of the transverse Hermitian structure and 
\[
c= \frac{1}{2}\sum_{i=1}^{2n} d\eta (\phi x_i, x_i)=\frac{1}{2}g(d\eta, F).
\]

\medskip
Since \(\nabla H\) is a 4-form, at \(p\) we have
\[
\begin{split}
\nabla_X \omega^\nabla(Y)
&= -\sum_{i=1}^n X(H(Y,x_i,\phi x_i)) + \sum_{i=1}^n H(\nabla_X Y, x_i, \phi x_i) \\
&= -\sum_{i=1}^n \nabla_X H(Y,x_i,\phi x_i) \\
&= \sum_{i=1}^n \nabla_Y H(X,x_i,\phi x_i) \\
&= -\nabla_Y \omega^\nabla(X),
\end{split}
\]
so \(\nabla \omega^\nabla\) is a 2-form.

\medskip
Next, compute \(\nabla_X \omega^\nabla(\xi)\):
\[
\nabla_X \omega^\nabla(\xi) = X(c ) = \frac{1}{2} X\, (g(F,d\eta))= \frac{1}{2} g(\nabla_X F,d\eta)+\frac{1}{2} g(F,\nabla_X d\eta)=0,
\]
since both \(F\) and \(d\eta\) are parallel (Step~1).

\medskip
Exploiting that $\omega^\nabla$ is a $2$-form, for any horizontal \(x\),
\[
0 = \nabla_\xi \omega^\nabla(x) = \xi(\phi\theta(x)) - \omega^\nabla(\nabla_\xi x).
\]
Observe that \(\nabla_\xi x\) is horizontal, and
\[
\nabla_\xi x = \nabla_x \xi + [\xi,x] + g^{-1} H(\xi,x) = [\xi,x] + (\iota_x d\eta)^\sharp.
\]
Hence,
\[
0 = \nabla_\xi \omega^\nabla(x) = \xi(\phi\theta(x)) - \phi\theta([\xi,x] + (\iota_x d\eta)^\sharp) = \mathcal{L}_\xi \phi\theta(x) + d\eta((\phi\theta)^\sharp, x).
\]
The first term vanishes because \(\xi\) preserves the transverse Hermitian structure, implying
\[
0 = d\eta((\phi\theta)^\sharp, x) \quad \text{for all horizontal } x,
\]
so \((\phi\theta)^\sharp = \phi(\theta^\sharp) \in \ker(d\eta)\). From equation \eqref{eqn:dn11}, also \(\theta^\sharp \in \ker(d\eta)\). Therefore,
\[
2\chi = \theta^\sharp - i \phi(\theta^\sharp) \in \ker(d\eta),
\]
as desired.
\end{proof}

\section{\texorpdfstring{$\nabla$}{}-flat manifolds}
In this section, we study the flatness condition for the connection $\nabla$ (defined in Theorem \ref{thm:nabla}). The main result of this section is the Sasaki analogue of the Bismut flat case investigated by Wang-Yang-Zheng \cite{WYZ}. 

\begin{defn}
Let $G$ be a Lie group with Lie algebra $\mathfrak{g}$. An almost contact structure $(\phi, \eta, \xi)$ is said to be \emph{left invariant} if $\xi \in \mathfrak{g}$, $\eta \in \mathfrak{g}^*$, and, for any $a \in G$, \, $ L_a \phi = \phi\, L_a,$ where $L_a$ denotes left multiplication by $a$.
\end{defn}

According to Morimoto~\cite{Mor}, any left invariant almost contact structure is equivalent to an almost contact structure defined on the Lie algebra $\mathfrak{g}$. In particular, a left invariant almost contact structure is said to be \emph{normal} if
\[
    N_\phi + d\eta \otimes \xi = 0
\]
on left invariant vector fields.

It was proved in \cite{Mor} that every compact connected Lie group admits a left invariant normal almost contact structure.
 
\begin{ex}  \label{ex:leftinv} (\cite{CoMa})
Let $G$ be a connected, compact Lie group of odd dimension, and let
$\mathfrak g$ be its Lie algebra. Since we will construct a left invariant
normal almost contact structure, it suffices to define it on $\mathfrak g$.

Let $T\subset G$ be a maximal torus with Lie algebra $\mathfrak t$.  Since $G$ has odd dimension, $\dim\mathfrak t$ is odd as well. We choose a system of positive roots
\[
\Delta=\{\alpha_1,\dots,\alpha_m\}\subset\mathfrak t^*.
\]
The complexified Lie algebra then decomposes as
\[
\mathfrak g^\C
   = \mathfrak t^\C 
     \oplus \bigoplus_{\alpha_j\in\Delta}\mathfrak s_{\alpha_j}
     \oplus \bigoplus_{\alpha_j\in\Delta}\mathfrak s_{-\alpha_j},
\]
where
\[
\mathfrak s_{\pm\alpha_j}
   =\{X\in\mathfrak g^\C : [H,X]=\pm 2\pi i\,\alpha_j(H)X,\ \forall H\in\mathfrak t\}.
\]
Each root space is one–dimensional and 
$\mathfrak s_{\alpha_j}=\overline{\mathfrak s_{-\alpha_j}}$.  
Choose $Z_j\in\mathfrak s_{\alpha_j}\setminus\{0\}$, so that
$\overline{Z_j}$ spans $\mathfrak s_{-\alpha_j}$.  Set
\[
X_{2j-1}=\frac{Z_j+\overline{Z_j}}{2},\qquad  
X_{2j}=\frac{\overline{Z_j}-Z_j}{2i},
\]
and define
\[
\mathfrak k := \operatorname{span}_{\R}\{X_1,\dots,X_{2m}\}.
\]
From the above decomposition one obtains the real direct sum
\[
\mathfrak g = \mathfrak t \oplus \mathfrak k.
\]

Let $\xi\in\mathfrak t$ be fixed, and choose a complementary subspace
$\mathfrak h$ so that $\mathfrak t=\R\xi\oplus\mathfrak h$.  
Let $I$ be an almost complex structure on $\mathfrak h$.

Define $\phi\colon\mathfrak g\to\mathfrak g$ by
\[
\phi(\xi)=0,\qquad
\phi|_{\mathfrak h}=I,\qquad
\phi(X_{2j-1})=X_{2j},\qquad
\phi(X_{2j})=-X_{2j-1},
\]
and let $\eta$ be the $1$-form satisfying
\[
\eta(\xi)=1,\qquad \eta|_{\mathfrak h\oplus\mathfrak k}=0.
\]
One checks easily that $(\phi,\xi,\eta)$ is a left invariant almost contact
structure on $G$.

\medskip

We now show that this structure is normal. Consider the product
\[
(G\times\R,\,J),
\]
where $J$ is defined as in \eqref{eqn:cpxstructure}. Writing $\partial_t$ for the coordinate vector field on~$\R$, by the definition of $J$ (see equation \eqref{eqn:cpxstructure}) we have that:
\[
J(\xi)=\partial_t,\qquad
J(\partial_t)=-\xi,\qquad
J|_{\mathfrak{h}}=\phi|_{\mathfrak{h}},\qquad
J|_{\mathfrak{k}}=\phi|_{\mathfrak{k}}.
\]
Since $\partial_t$ spans $\mathrm{Lie}(\R)$ and $\phi$ is left invariant, the tensor $J$ defines a left invariant almost complex structure on $\mathrm{Lie}(G\times\R)=\g\oplus\R$. \par

Its $(1,0)$-subspace is
\[
(\mathfrak g\oplus\R)^{1,0}
=
\operatorname{span}_{\C}\!\left\{
\frac{\xi-i\partial_t}{2},\,
Z_1,\dots,Z_m
\right\}
\;\oplus\;
\mathfrak h^{1,0},
\]
where $\mathfrak h^{1,0}$ denotes the $+i$-eigenspace of $I$.
The usual bracket relations for root vectors, together with the fact that
$\mathfrak h\subset\mathfrak t$ is abelian, imply that this subspace is
involutive.  Therefore $J$ is integrable, and the almost contact structure
$(\phi,\xi,\eta)$ is normal.

\end{ex}
Next, we prove that any left invariant normal almost complex structure on a compact connected Lie group is constructed as in the Example \ref{ex:leftinv}. To prove this claim we will make use of the fundamental result of Pittie \cite{Pit} which shows that any left invariant complex structure on a compact connected Lie group is a Samelson complex structure \cite{Sam}.  \par

\medskip

Let $G$ be a compact Lie group endowed with a left-invariant normal almost contact structure $(\phi, \xi, \eta)$. We decompose $\mathfrak g = \mathrm{Lie}(G)$ as  
\[
\mathfrak g = \R\xi \oplus \ker(\eta),
\]  
where 
\[
\ker(\eta):=\{ X \in \g \ | \ \eta(X)=0\}.
\]
Consider the associated complex structure $J$ on $G \times \mathbb{R}$ (see Equation \eqref{eqn:cpxstructure}). Since $\phi$ is left-invariant, $J$ is left-invariant as well.  Furthermore, we may regard $J$ as a left-invariant complex structure on the compact Lie group $S^1 \times G$,  as $
\mathrm{Lie}(G \times \mathbb{R}) \simeq \mathrm{Lie}(G \times S^1)
$. By the fundamental result of Pittie~\cite{Pit}, any such structure must be a Samelson complex structure, i.e., $J$ is uniquely determined by a maximal torus, an almost complex structure on it, and a choice of positive roots.  

Let $\tilde{\mathfrak{g}} = \mathfrak{g} \oplus \mathfrak{u}(1)$ denote the Lie algebra of $G \times S^1$. Any maximal torus $\tilde{\mathfrak{t}} \subset \tilde{\mathfrak{g}}$ necessarily decomposes as  
\[
\tilde{\mathfrak{t}} = \mathfrak{t} \oplus \mathfrak{u}(1),
\]  
where $\mathfrak{t} \subset \mathfrak{g}$ is a maximal torus in $\mathfrak g$. A choice of a system of positive roots for $\tilde{\mathfrak{g}}$ with respect to $\tilde{\mathfrak{t}}$ may then be regarded as a system of positive roots for $\mathfrak g$ with respect to $\mathfrak{t}$ extended trivially over $\mathfrak{u}(1)$.  \par \medskip
Recall that $J$ on $\tilde \g=\g \oplus \mathfrak{u}(1)$ is defined as follows:
\[
J(X, a \partial_t)=(\phi(X)-a \xi, \eta(X) \partial_t)
\]
for any $X \in \g$ and $a \in \R$. From the definition of $J$ it is immediate to note that 
\[
\ker(\eta)=\{X \in \g \ | \ J X \in \g\},
\]
and that $J=\phi$ on $\ker(\eta)$. \par 
\medskip
Since $\partial_t \in \tilde{\mathfrak{t}}$ and $J$ preserves the maximal torus $\tilde{\mathfrak{t}}$, it follows that  
\[
\xi=-J \partial_t  \in \tilde{\mathfrak{t}} \cap \mathfrak{g} = \mathfrak{t}.
\]  
Define  
\[
\mathfrak{h} := \ker(\eta) \cap \mathfrak{t} = \{ X \in \mathfrak{t} \mid J X \in \mathfrak{t} \},
\]  
so that we obtain the decomposition  
\[
\mathfrak{t} = \R \xi \oplus \mathfrak{h}.
\]  

Since $\mathfrak h \subset \ker(\eta)$, we have $\phi = J$ on $\mathfrak h$. In particular, $\phi$ defines an almost complex structure on $\mathfrak h$.  

Let $\Delta = \{\alpha_1, \dots, \alpha_m\}$ denote the positive roots associated to $J$. For each root $\alpha_j$, choose a nonzero generator $Z_j \in \mathfrak s_{\alpha_j} \subset \mathfrak g^{\mathbb{C}}$, and let $\overline{Z_j}$ denote its complex conjugate. Define the real vectors  
\[
X_{2j-1} = \frac{Z_j + \overline{Z_j}}{2}, \qquad X_{2j} = \frac{\overline{Z_j} - Z_j}{2i}.
\]  
Since $J$ is a Samelson complex structure,  
\[
J X_{2j-1} = X_{2j}, \qquad J X_{2j} = -X_{2j-1}.
\]  

Finally, set  
\[
\mathfrak k = \mathrm{span}\{ X_1, \dots, X_{2m} \} \subset \ker(\eta).
\]  
Then $\phi = J$ on $\mathfrak k$, and $\eta$ is equal to $1$ on $\xi$ and zero on $\mathfrak h$ and $\mathfrak k$. \par

Since \[\mathfrak g = \mathfrak t \oplus \mathfrak k = \mathbb{R}\xi \oplus \mathfrak h \oplus \mathfrak k,
\]  
$(\phi, \eta, \xi)$ is constructed exactly as in Example~\ref{ex:leftinv}.
Hence we have proven the following
\begin{thm} \label{thm:left-inv_acm}
    Let $G$ be a compact Lie group and let $(\phi, \eta, \xi)$ be a left-invariant normal almost contact structure. 
Then $(\phi, \eta, \xi)$ is constructed as in Example~\ref{ex:leftinv}. 
In particular, any left-invariant normal almost contact structure is completely determined by:
\begin{itemize}
    \item a choice of a maximal torus $\mathfrak{t} \subset \mathfrak g$,
    \item a choice of a vector $\xi \in \mathfrak{t}$,
    \item a complementary subspace $\mathfrak{h} \subset \mathfrak t$ such that $\mathfrak t = \mathbb{R}\xi \oplus \mathfrak{h}$,
    \item a choice of an almost complex structure on $\h$,
    \item a choice of a system of positive roots for $\mathfrak g$ with respect to $\mathfrak t$.
\end{itemize}
\end{thm}
\begin{rmk}
Assume that $G^{2n+1}$ is a compact Lie group endowed with a left-invariant normal almost contact structure $(\phi, \eta, \xi)$. Then, $(\phi, \eta, \xi)$ is a contact structure, that is $\eta \wedge d\eta^n$ is a volume form, if and only if $\mathrm{rank}(G)=1$, i.e., if and only if $G=\mathrm{SU}(2), \mathrm{SO}(3),\R$. \par
In fact, let $H_i \in \mathfrak{t}$, then 
\[
 d \eta(H_1,H_2)=-\eta([H_1,H_2])=0,
\]
implying that $ \R \xi \subset \mathfrak{t} \subset \ker(d\eta)$. Since $\eta \wedge d\eta^n$ is a volume form if and only if $\R \xi= \ker(d\eta)$, we must have that $ \R \xi = \mathfrak{t}$. 
\end{rmk}
It is a classical result that any even dimensional compact Lie group endowed with a bi-invariant metric $b$ admits a left-invariant complex structure that makes $b$ Hermitian \cite{Sam}. Here we prove an odd-dimensional counterpart of that result.
\begin{prop} \label{prop:bi-inv}
    Let $G$ be a compact Lie group of odd dimension and let $b$ be a bi-invariant metric on $G$. Then, there exists a left-invariant normal almost contact structure $(\phi,\eta,\xi)$ such that $b$ is compatible with $(\phi,\eta,\xi)$.
\end{prop}
\begin{proof}
Let $G$ be a compact Lie group of odd dimension $\dim G = 2n+1$, and let $b$ be a bi-invariant metric on $G$.  
Let $T \subset G$ be a maximal torus with Lie algebra $\mathfrak{t}$.  
Choose a system of positive roots $\Delta = \{\alpha_1, \dots, \alpha_m\}$. 

The complexified Lie algebra $\mathfrak{g}^\C$ splits as
\[
\mathfrak{g}^\C = \mathfrak{t}^\C \oplus \bigoplus_{\alpha \in \Delta} (\mathfrak{s}_\alpha \oplus \mathfrak{s}_{-\alpha}).
\]

Extend $b$ to a $\C$-bilinear form on $\mathfrak{g}^\C$. Then we have the following orthogonality relations:
\[
b(\mathfrak{t}^\C, \mathfrak{s}_{\pm \alpha}) = 0, \quad
b(\mathfrak{s}_\alpha, \mathfrak{s}_\beta) = 0 \text{ if } \alpha \neq -\beta, \quad
b(\mathfrak{s}_\alpha, \mathfrak{s}_{-\alpha}) \neq 0.
\]

For each positive root $\alpha_j$, choose a non-zero generator $Z_j \in \mathfrak{s}_{\alpha_j}$ such that
\[
b(Z_j, \overline{Z}_j) = 1.
\]

Define
\[
X_j := \frac{1}{\sqrt{2}} (Z_j + \overline{Z}_j), \quad
Y_j := \frac{1}{\sqrt{2}i} (\overline{Z}_j - Z_j).
\]

Then $\{X_j, Y_j\}$ is an orthonormal basis of $\mathfrak{k}_j := \mathfrak{s}_{\alpha_j} \oplus \mathfrak{s}_{-\alpha_j}$ with respect to $b$.  
Set
\[
\mathfrak{k} := \mathrm{span}\{X_1, Y_1, \dots, X_m, Y_m\},
\]
so that we have an orthogonal decomposition
\[
\mathfrak{g} = \mathfrak{t} \oplus \mathfrak{k}.
\]

Let $\{H_0, H_1, \dots, H_{2n-2m}\}$ be an orthonormal basis of $\mathfrak{t}$, and set
\[
\xi := H_0, \quad \eta := H_0^\flat.
\]
Then
\[
\{\xi, H_1, \dots, H_{2n-2m}, X_1, Y_1, \dots, X_m, Y_m\}
\]
is an orthonormal basis of $\mathfrak{g}$ with respect to $b$.

Define the endomorphism $\phi$ by
\[
\phi(\xi) = 0, \quad \phi(H_{2i-1}) = H_{2i}, \quad \phi(H_{2i}) = -H_{2i-1}, \quad i = 1, \dots, n-m,
\]
\[
\phi(X_j) = Y_j, \quad \phi(Y_j) = -X_j, \quad j = 1, \dots, m.
\]

By Theorem \ref{thm:left-inv_acm}, $(\phi, \xi, \eta)$ is a normal left-invariant almost contact structure on $G$, and $b$ is compatible with $(\phi, \xi, \eta)$, that is, 
\[
b(\phi X, \phi Y) = b(X, Y) - \eta(X) \eta(Y), \quad \forall X, Y \in \mathfrak{g}.
\]

\end{proof}

\begin{rmk} \label{rmk:3}
Let $G$ be a simply connected odd-dimensional Lie group endowed with a bi-invariant metric $b$. 
By a classical result of Milnor~\cite{Mil}, one has a decomposition
\[
(G,g) = (G' \times \mathbb{R}^k,\; b' + g_E),
\]
where $G'$ is simply connected, compact and semisimple, $b'$ is a bi-invariant metric on $G'$, and $g_E$ denotes the Euclidean metric on~$\mathbb{R}^k$. 
Since 
\[
\mathrm{Lie}(G' \times \mathbb{R}^k) = \mathrm{Lie}(G' \times T^k),
\]
the product $G' \times T^k$ can be equipped with the same bi-invariant metric $b = b' + g_E$.

Because $G' \times T^k$ is compact, Proposition~\ref{prop:bi-inv} ensures the existence of a left-invariant normal almost contact structure $(\phi,\xi,\eta)$ such that $(\phi,\xi, \eta, b)$ is a normal almost contact metric structure. 
As $(\phi,\xi, \eta, b)$ is left-invariant, it is completely determined by its values on the Lie algebra, and therefore it is defined naturally on $G' \times \mathbb{R}^k$.

We conclude that any simply connected odd-dimensional Lie group endowed with a bi-invariant metric $g$ admits a normal left-invariant almost contact structure $(\phi,\xi,\eta)$ for which $(\phi,\xi,\eta, b)$ is a normal almost contact metric structure.
\end{rmk}

Let $G$ be a odd dimensional Lie group endowed with a left-invariant normal almost contact metric structure $(\phi,\xi,b)$ such that $b$ is bi-invariant. 
 In particular, since $\xi$ is left invariant, $\xi$ is Killing. The Friderich-Ivanov connection $\nabla$ is therefore well defined. We now make the expression of $\nabla$ explicit.  \par
Let $\overline{\nabla}$ be the connection for which $\overline{\nabla}X=0$ for any left-invariant vector field $X$. From this, it is immediate to observe that $\overline{\nabla}b=\overline{\nabla}\phi=\overline{\nabla}\xi=0$. On left invariant vector fields
\[
b(T^{\overline{\nabla}}(X,Y),Z)=-b([X,Y],Z) \in \Lambda^3(\g^*). 
\]
By uniqueness, $\nabla=\overline{\nabla}$ and $H(X,Y,Z)=-b([X,Y],Z)$ on left-invariant vector fields. \par
Since $H$ is a bi-invariant $3$-form, it follows that $dH=0$. Also, using that the left-invariant vector fields are parallel we have that the curvature of $\nabla$ is identically zero.  \par
\medskip
\medskip
The aim of this section is to characterize compact manifolds that are $\nabla$-flat.  
While the Riemannian structure of such manifolds is well understood \cite{CS1,CS2,AF}, our goal here is to study the 
compatible normal almost contact structure.

\medskip

Let $G$ be a simply connected odd-dimensional Lie group endowed with a left-invariant normal almost contact metric structure $(\phi,\xi,b)$ such that $b$ is bi-invariant.  
As we have already observed, the associated Friedrich-Ivanov connection $\nabla$ is flat.

\medskip

By a result of Milnor~\cite{Mil}, we can write
\[
(G,b) \cong (G' \times \mathbb{R}^k, \; b' + g_E),
\]
where $G'$ is simply connected, compact, and semisimple, $b'$ is bi-invariant on $G'$, and $g_E$ denotes the Euclidean metric on $\mathbb{R}^k$.  

Consider a group homomorphism
\begin{equation} \label{eqn:rho}
\rho: \mathbb{Z}^k \to \mathrm{Isom}(G',b')
\end{equation}
with associated $\mathbb{Z}^k$-action on $G = G' \times \mathbb{R}^k$ defined, for any $n \in \mathbb{Z}^k$, by
\begin{equation} \label{eqn:zaction}
n \cdot (p,t) = (\rho(n) p, t+n).
\end{equation}

By construction, the action \eqref{eqn:zaction} is by isometries of $b$.  
Therefore, if it also preserves $\xi$ and $\phi$, the quotient manifold $
M = G / \mathbb{Z}^k $ 
inherits a normal almost contact metric structure $(\phi,\xi,b)$ for which the connection $\nabla$ is flat.

\medskip

Inspired by the analogous result in~\cite{WYZ} (the first part of the proof is simplified here due to Theorem~\ref{thm:CKL} for $\nabla$, which was not yet known for the Bismut connection when \cite{WYZ} was written), we obtain the following characterization:

\begin{thm} \label{thm:flat}
Let \((M, \phi, \xi, g, \nabla)\) be a compact Sasaki with torsion manifold.  
Then the connection $\nabla$ is flat if and only if, up to a finite cover, $M$ can be written as
\[
M = G / \mathbb{Z}^k,
\]
where $G$ is a simply connected odd-dimensional Lie group equipped with a left-invariant normal almost contact metric structure whose metric is bi-invariant.  
Furthermore, the action of $\mathbb{Z}^k$ on $G$, given as in \eqref{eqn:zaction} for some homomorphism $\rho$ as in \eqref{eqn:rho}, preserves the left-invariant almost contact structure on $G$. 
\end{thm}

\begin{proof}
Only the implication from left to right remains to be proved, as the reverse direction has already been addressed. \par
Around each point $p \in M$ there exists a \emph{local} parallel orthonormal frame
\[
\{\xi, e_1, \phi(e_1), \dots, e_n, \phi(e_n)\},
\]
such that the vector fields $e_i, \phi(e_i)$ are horizontal. 
This is a standard consequence of the flatness of $\nabla$ and the parallelism conditions $\nabla \xi = \nabla \phi = 0$\footnote{While the Reeb vector field $\xi$ is globally defined on $M$, the vector fields $e_i, \phi(e_i)$ are, in general, only local.}. \par

Since $\nabla e_i = \nabla \phi(e_i) = \nabla \xi = 0$ in this local frame, for any local vector fields $X,Y,Z \in \{\xi, e_i, \phi(e_i)\}$ we have
\begin{equation} \label{eqn:gH_local}
g([X,Y],Z) = -H(X,Y,Z),
\end{equation}
where $H$ is the torsion of $\nabla$. 

By Theorem~\ref{thm:CKL}, since $\nabla$ is flat, the torsion $H$ is parallel. 
Hence, for any local vector field $W$, we have
\begin{equation} \label{eqn:torsion_constant_local}
W \big(H(X,Y,Z)\big) = \nabla_W H(X,Y,Z) = 0,
\end{equation}
showing that $H(X,Y,Z)$ is constant within the domain of the chosen frame. \par

Furthermore, for such local vector fields, the complete skew-symmetry of the torsion $H$ implies
\begin{equation} \label{eqn:biinv_local}
g([X,Y],Z) = -H(X,Y,Z) = -H(Y,Z,X) = g([Y,Z],X) = g(X,[Y,Z]).
\end{equation}
This relation shows that $g$ behaves like a bi-invariant metric on the local Lie algebra generated by the frame $\{\xi, e_1, \phi(e_1), \dots, e_n, \phi(e_n)\}$. \par

Equations~\eqref{eqn:gH_local}, \eqref{eqn:torsion_constant_local}, and \eqref{eqn:biinv_local} imply that, these vector fields are the left-invariant vector fields of a local Lie group endowed with a left-invariant normal almost contact metric structure and a bi-invariant metric. 

Lifting the metric and the almost contact structure to the universal cover $\tilde M$, we obtain that $\tilde M$ is a simply-connected odd-dimensional Lie group equipped with a left-invariant normal almost contact metric structure whose metric is bi-invariant. \par
The final part of the proof follows by the same argument as in \cite[Theorem 1]{WYZ}.
\end{proof}

\section{\texorpdfstring{$\nabla$}{}- Einstein manifolds}

Let \((M, \phi, \xi, g, \nabla)\) be a Sasaki with torsion manifold and let $R^\nabla$ be the curvature tensor of $\nabla$.

Since $\nabla$ is a metric connection preserving the tensor $\phi$, its curvature satisfies
\[
R^\nabla \in \Lambda^2(T^*M)\otimes \Lambda^{1,1}(\mathcal{F}^\perp_\xi).
\]

The presence of the endomorphism $\phi$ naturally induces more contractions of $R^\nabla$. The first one is the Ricci tensor of $\nabla$, defined by 
\[
\mathrm{Ric}^\nabla(X,Y)=\sum_{i=1}^{2n+1} R^\nabla(e_i,X,Y,e_i),
\]
where $\{e_i\}$ is a local orthonormal frame. It is well known that
\[
\mathrm{Ric}^\nabla(X,Y)
= \mathrm{Ric}^g(X,Y) - \frac{1}{4} H^2(X,Y) - \frac{1}{2} (\delta H)(X,Y),
\]
where $H^2(X,Y)=g(\iota_X H,\iota_Y H)$ and $\delta$ denotes the codifferential with respect to $g$.

A second contraction, intrinsically associated with $\phi$, is given by
\[
\rho^\nabla(X,Y)=\frac{1}{2}\sum_{i=1}^{2n+1} R^\nabla(X,Y,e_i,\phi e_i).
\]

The tensor $\rho^\nabla$ plays the role of a $\phi$-Ricci form. In particular, its vanishing imposes a strong restriction on the holonomy of $\nabla$: in fact, $\rho^\nabla=0$ if and only if $\mathrm{Hol}(\nabla)\subseteq \mathrm{SU}(n).$

\medskip 
The following lemma will be useful for subsequent considerations.

\begin{lem} \label{lem:curvature}
Let \((M, \phi, \xi, g, \nabla)\) be a Sasaki with torsion manifold. Then, for any horizontal vector fields $x,y,z,w$ the following identities hold:
\begin{align}
 R^\nabla(x,y,z,w)=&  \ R^B(x,y,z,w)+d\eta(x,y)\cdot d\eta(z,w);\\
 R^\nabla(\xi,x,y,z)=& -\nabla_x d\eta(y,z). 
\end{align}
\end{lem}
\begin{proof}
The curvature tensor $R^{\nabla}$ can be expressed as follows: 
\begin{align*}
  R^{\nabla}(X,Y,Z,W)&=R^g(X,Y,Z,W)+\frac{1}{2} {\nabla}_X {H} (Y,Z,W) -\frac{1}{2} {\nabla}_Y {H} (X,Z,W) \\& +\frac{1}{2} g({H}(X,Y), {H}(Z,W)) +\frac{1}{4} g({H}(Y,Z), {H}(X,W))+\frac{1}{4} g({H}(Z,X), {H}(Y,W)),
\end{align*}
where $R^g$ is the curvature of the Levi-Civita connection. Using the formula above, the fact that $x,y,z,w$ are horizontal vector fields, and the O'Neill identities relating the curvature tensors $R^g$ and $R^{g^T}$ (see, for instance \cite[Theorem~2.5.16]{BG}), we obtain:
\begin{align*}
  R^{{\nabla}}(x,y,z,w)&=R^g(x,y,z,w)+\frac{1}{2} 
  \nabla^B_x {d^\phi F} (y,z,w) -\frac{1}{2} \nabla^B_y {d^\phi F} (x,z,w) \\& +\frac{1}{2} g^T(d^\phi F(x,y), d^\phi F(z,w)) +\frac{1}{4} g^T(d^\phi F(y,z), d^\phi F(x,w))+\frac{1}{4} g^T(d^\phi F(z,x), d^\phi F(y,w))\\&
  +\frac{1}{2} d\eta(x,y) \cdot d\eta(z,w)+\frac{1}{4} d\eta(x,w) \cdot d\eta(y,z)+\frac{1}{4} d\eta(z,x) \cdot d\eta(y,w)\\&
  +\frac{1}{2} d\eta (x,y) \cdot d\eta(z,w) +\frac{1}{4} d\eta (x,z) \cdot d\eta(y,w)-\frac{1}{4} d\eta(x,w) d\eta(y,z)\\
  &= R^B(x,y,z,w) + d\eta (x,y) \cdot d\eta(z,w). 
\end{align*}
\medskip
We now prove the second statement of the lemma. Using the fact that $\mathcal{L}_\xi \nabla=0$, we have
\begin{align*}
R^\nabla(\xi,x,y,z)=&\ g(\nabla_\xi \nabla_x y-\nabla_x \nabla_\xi y-\nabla_{[\xi,x]} y,z)\\
=&\ g([\xi, \nabla_x y],z)-g(\nabla_x [\xi,y],z)-g(\nabla_{[\xi,x]}y,z)+d\eta( \nabla_x y, z)-(\nabla_x \iota_y d\eta) z\\
=&\ \mathcal{L}_\xi \nabla (x,y,z)-\nabla_x d\eta (y,z)\\
=&-\nabla_x d\eta (y,z).
\end{align*}
This follows from the fact that $\nabla$ is entirely determined by $g$, $\phi$, and $\eta$, and $\xi$ satisfies
\[
\mathcal{L}_\xi g = \mathcal{L}_\xi \phi = \mathcal{L}_\xi \eta = 0.
\]
\end{proof}

Using the lemma above, one can show that $\rho$ can be expressed as follows:
\begin{prop} \label{prop:rhoB}
Let $(M,\phi,\xi,g, \nabla)$ be a Sasaki with torsion manifold. Then \[\rho^\nabla = \rho^B - d(c \, \eta)\]
is a closed $2$-form. Moreover, $\rho^\nabla = 0$ if and only if, the function $c$, defined by 
\[
c := \tfrac12\, g(d\eta, F),
\]
is constant and
\[
c\, d\eta = \rho^B,
\]
where $\rho^B$ denotes the Bismut Ricci form of the transverse Hermitian structure.\label{thm:BHE}
\end{prop}
\begin{proof}
The proof of this theorem is a direct consequence of Lemma \ref{lem:curvature}. First, notice that to simplify the computations, we may take $\{e_i\}=\{\xi, x_i, \phi x_i\}$ as a $\phi$-basis. Since $\phi \xi = 0$, we have
\[
\rho^\nabla(X,Y) = \frac{1}{2} \sum_{i=1}^{2n} R^\nabla (X,Y,x_i, \phi x_i) = 0.
\]

Using the first assertion of Lemma \ref{lem:curvature}, for any horizontal vector fields $x$ and $y$, we get
\[
\rho^\nabla(x,y) = \frac{1}{2} \sum_{i=1}^{2n} R^\nabla (x,y,x_i, \phi x_i) = \rho^B(x,y) - c \, d\eta(x,y). 
\]

Using the second assertion instead, we obtain
\begin{align*}
\rho^\nabla(\xi,x) =& \frac{1}{2} \sum_{i=1}^{2n} R^\nabla (\xi,x,x_i, \phi x_i) = - \frac{1}{2} \sum_{i=1}^{2n} \nabla_x d\eta (x_i, \phi x_i) \\=
&\frac{1}{2} g(\nabla_x d\eta, F) = \frac{1}{2} x \big(g(\nabla_x d\eta, F)\big) = dc(x). 
\end{align*}

Therefore, 
\[
\rho^\nabla = \rho^B - d(c \, \eta).
\]

The theorem then follows immediately.
\end{proof}

\begin{rmk}
Proposition \ref{thm:BHE} provides a convenient expression for $\rho^\nabla$, which can be used to construct examples of Sasaki with torsion manifolds whose holonomy is contained in $\mathrm{SU}(n)$, in the spirit of \cite{GGP}.

Let $(N,J,\omega)$ be a $4$-dimensional Kähler--Einstein manifold with positive Einstein constant $\lambda$, normalized so that $\lambda=1$. In particular, $[\frac{\omega}{2\pi}] \in H^2(N,\mathbb{Z})$. Let $c= \pm \sqrt{2}$, and set
\[
\alpha := \frac{c}{2}\,\omega.
\]

Setting $k = \frac{c}{2}$, we have
\[
\frac{\alpha}{2\pi k} = \frac{\omega}{2\pi},
\]
hence $\left[\frac{\alpha}{2\pi k}\right] \in H^2(N,\mathbb{Z})$. Therefore, there exists a principal $S^1$-bundle $\pi:M\to N$ with a connection $1$-form $\eta$ such that
\[
d\eta = \alpha = \frac{c}{2}\,\omega,
\]
where $S^1 = \mathbb{R}/(2\pi k)\mathbb{Z}$ is the circle with rescaled period. \par

The induced Sasaki structure with torsion on $M$ satisfies
\[
\rho^\nabla = \rho^{\mathrm{LC}} - c\, d\eta 
= \omega - \frac{c^2}{2}\,\omega = 0.
\]

This construction applies, for instance, to $\mathbb{CP}^1 \times \mathbb{CP}^1$, $\mathbb{CP}^2$, and to the blow-up of $\mathbb{CP}^2$ at $r$ points, for $3 \le r \le 8$. In these cases, the total space $M$ is diffeomorphic to $S^2 \times S^3$, $S^5$, and $\#_r(S^2 \times S^3)$, respectively.
\end{rmk}

\medskip

We now introduce the following notion.
\begin{defn}
    We say that a Sasaki with torsion structure $(M,\varphi,\xi,\eta,g)$ is strong if the torsion form of $\nabla$, given by $H=\eta\wedge d\eta+ d^{\varphi}F$, is closed.
\end{defn}

\begin{defn}
Let $(M,\phi,\xi,g, \nabla)$ be a strong Sasaki with torsion manifold. We say that $(M,\phi,\xi,g)$ is \emph{$\nabla$-Einstein} if $\rho^\nabla=0$.
\end{defn}

The latter condition is particularly relevant in the context of string compactifications~\cite{PT}, and also for its applications in Hermitian geometry, which motivate the terminology, as explained in the following remark.
 \begin{rmk} \label{rmk:oddeven}
Let $(M,\phi,\xi,g,\nabla)$ be a $\nabla$-Einstein manifold. We recall that
\[
H = \eta \wedge d\eta + d^\phi F.
\]
The product $M \times \mathbb{R}$ carries a natural Hermitian structure defined by
\[
g_t = dt^2 \oplus g, \qquad J(\xi) = \partial_t, \qquad J(\partial_t) = -\xi, \qquad J|_{\ker(\eta)} = \phi.
\]
The associated fundamental form $\omega = g_t(\cdot, J\cdot)$ is given by
\[
\omega = dt \wedge \eta + F.
\]
A direct computation shows that the torsion of the associated Bismut connection is
\[
d^c \omega = J(-dt \wedge d\eta + dF) = \eta \wedge d\eta + d^\phi F = H,
\]
and hence $dH = 0$.

Moreover, the Bismut connection of $(g_t,J)$ splits as $\nabla^{\mathrm{LC}}_{\mathbb{R}} \oplus \nabla$, where $\nabla^{\mathrm{LC}}_{\mathbb{R}}$ denotes the Levi--Civita connection of the flat metric on $\mathbb{R}$. In particular, its curvature reduces to that of $\nabla$. Since $\rho^\nabla = 0$, it follows that the Bismut connection of $(g_t,J)$ also has vanishing Bismut Ricci form. Therefore, $(g_t,J)$ is Bismut--Hermite--Einstein.

The same argument applies if $\mathbb{R}$ is replaced by $S^1$, with $\partial_t$ replaced by the standard Reeb vector field on $S^1$.
\end{rmk}
In view of Proposition \ref{thm:BHE}, we have that a Sasaskian with torsion manifold $(M,\phi,\xi,g, \nabla)$ is $\nabla$-Einstein if and only if, setting
\[
c := \tfrac12\, g(d\eta, F),
\]
the following conditions are fulfilled
\begin{enumerate}
    \item the function $c$ is constant,
    \item $c\, d\eta = \rho^B,$ where $\rho^B$ denotes the Bismut Ricci form of the transverse Hermitian structure,
    \item $dH=d\eta \wedge d\eta+ dd^\phi F=0$.
\end{enumerate}
Note that the equality $c\, d\eta = \rho^B,$ implies that $c^2=b=\text{const}$, where $b$ is the Bismut scalar curvature, that is, $b=\tfrac{1}{2} g(\rho^B, \omega)$. \par

\medskip
In dimension 3, the vanishing of $\rho^\nabla$ already implies the vanishing of the full curvature tensor, since $\nabla \xi=0$. Consequently, by Theorem \ref{thm:flat}, the manifold is, up to its universal cover, either $\R^3$ or $\mathrm{SU}(2)$.

Here, we present two examples of non-compact $\nabla$-Hermitian-Einstein structures. These are local examples defined on suitable open subsets of $\mathbb{R}^5$ and $\mathbb{R} \times \mathbb{C}^2  \times \mathbb{C}$, respectively. We then turn our attention to the compact case, focusing in particular on the five dimensional case.

\begin{ex} \label{ex:noncompactBHE}
Let $\mathbb{R}^5$ be endowed with coordinates $(x_1,y_1,x_2,y_2,z)$, and define the vector fields
\[
X_1 = \partial_{x_1} + \frac{y_1}{2}\,\partial_z, \quad 
X_2 = \partial_{x_2} - \frac{y_2}{2}\,\partial_z, \quad 
Y_1 = \partial_{y_1} - \frac{x_1}{2}\,\partial_z, \quad 
Y_2 = \partial_{y_2} + \frac{x_2}{2}\,\partial_z.
\]
Consider the $1$-form
\[
\eta = dz + \frac{1}{2}\bigl(x_1 dy_1 - x_2 dy_2 - y_1 dx_1 + y_2 dx_2\bigr),
\]
and define a normal almost contact structure $(\phi,\xi,\eta)$ by
\[
\xi = \partial_z, \qquad \phi(\partial_z) = 0, \qquad 
\phi(X_i) = Y_i, \qquad \phi(Y_i) = -X_i.
\]
Let $k>0$, and define a metric $g$ on
\[
B_k=\{(x_1,y_1,x_2,y_2,z)\in \mathbb{R}^5 \mid x_1^2+y_1^2+x_2^2+y_2^2<4k\}
\]
by
\[
g = \eta^2 + u \sum_{i=1}^2 \bigl((dx_i)^2 + (dy_i)^2\bigr),
\]
where
\[
u = k - \frac{1}{4}(x_1^2 + y_1^2 + x_2^2 + y_2^2).
\]
Then $(\phi,\xi,\eta,g)$ defines a normal almost contact metric structure on $B_k$, and the Reeb vector field $\xi$ is Killing. Hence the structure is Sasaki with torsion, with fundamental $2$-form
\[
F = u \cdot(dy_1 \wedge dx_1 + dy_2 \wedge dx_2) = u\cdot F_E.
\]
We claim that this structure is $\nabla$-Hermitian-Einstein. First, observe that
\[
d\eta = dx_1 \wedge dy_1 - dx_2 \wedge dy_2,
\]
and therefore $g(d\eta,F)=0$, which implies $c=0$.

Let $(g^T,F) = (u\,g_E, u\,F_E)$ denote the transverse Hermitian structure. Since $(g^T,F)$ is conformal to the flat K\"ahler structure $(g_E,F_E)$ and the Bismut Ricci form $\rho^B$ is conformally invariant in real dimension $4$, it follows that $\rho^B = 0$. Hence $\rho^\nabla = 0$ by Proposition~\ref{thm:BHE}.

It remains to check that $dH = 0$. A direct computation gives
\begin{align*}
dH &= d\eta \wedge d\eta + dd^\phi F \\
&= d\eta \wedge d\eta + (dd^\phi u)\wedge F_E \\
&= -2\vol_E - \sum_{i} (\partial_{x_i}^2u+\partial_{y_i}^2u )\vol_E  \\     &=(-2+2)\vol_E=0.
\end{align*}

Finally, the above structure is not $\nabla$-flat and, in particular, does not have $\nabla$-parallel torsion. Indeed, both of the claims above follow since $d\eta$ has non-constant norm.  \par

\end{ex}
The next example is constructed similarly to the previous one, so we only sketch the details. 
\begin{ex}
We consider \(\mathbb{C}^2 \times \mathbb{C}\). Let \(x,y\) be the coordinates on \(\mathbb{C}^2\), and let \(z\) be the coordinate on \(\mathbb{C}\).
Let
\[
g = u \,(dx \otimes d\bar{x} + dy \otimes d\bar{y}) + \gamma \otimes \bar{\gamma}
\]
be a Hermitian metric on \(\mathbb{C}^2 \times \mathbb{C}\), where \(\gamma := dz + \bar{x} \,dy\) and $u$ is a positive function.

The associated fundamental form \(F\) is given by
\[
F = -iu \,(dx \wedge d\bar{x} + dy \wedge d\bar{y}) - i\,\gamma \wedge \bar{\gamma}.
\]

We show that \(\rho^B = 0\). This manifold can be viewed as a \(\mathbb{C}\)-bundle over \(\mathbb{C}^2\) with connection forms
\[
\alpha = \frac{\gamma + \bar{\gamma}}{2},
\qquad
\beta = \frac{\gamma - \bar{\gamma}}{2i}.
\]
From this perspective, the metric can be written as
\[
g = \pi^*(u\, g_E) + \gamma \otimes \bar{\gamma},
\]
where \(\pi\) denotes the projection onto \(\mathbb{C}^2\), and \(g_E\) is the Euclidean metric on \(\mathbb{C}^2\).

Following \cite{GGP}, we have
\[
\rho^B = \pi^* \rho^B(u \, g_E) - d(c_\alpha \,\alpha) -d(c_\beta \,\beta),
\]
where \(\rho^B(u \, g_E)\) is the Bismut Ricci form of \(u \, g_E\), and \(c_\alpha\), \(c_\beta\) denote the traces of \(d\alpha\), \(d\beta\) with respect to \( u\, \omega_E\).

Since \(d\gamma = -dy \wedge d\bar{x}\), it follows that both \(c_\alpha\) and \(c_\beta\) vanish. Moreover, since \(u \,g_E\) is conformal to the Euclidean metric, we have \(\rho^B(u\,g_E) = 0\). Hence, \(\rho^B = 0\).

We use this observation to construct new examples of \(\nabla\)-Einstein manifolds in dimension \(7\). \par \medskip

Let \(k>0\) and define
\[
B_k := \{(t,x,y,z) \in \mathbb{R} \times \mathbb{C}^2 \times \mathbb{C} \ | \ |x|^2 + |y|^2 < k \}.
\]
Set
\[
u = k - (|x|^2 + |y|^2),
\qquad
\gamma = dz - \bar{x}\,dy.
\]

We define a Sasaki with torsion structure \((\phi,\xi,g,\nabla)\) on \(B_k\) as follows. Let
\[
\eta = dt + \frac{i}{2}\big(x \, d\bar{x} - \bar{x}\, dx + y\, d\bar{y} - \bar{y}\, dy\big),
\qquad
\xi = \partial_t,
\]
and
\[
g = \eta^2 + u\,(dx \otimes d\bar{x} + dy \otimes d\bar{y}) + \gamma \otimes \bar{\gamma}.
\]
Note that $u>0$ on $B_k$, so \(g\) is positive definite.

The endomorphism \(\phi\) is defined with respect to the local frame
\[
\{\partial_t,\ \partial_x,\ \bar{x}\partial_z+\partial_y,\ \partial_z,\ \partial_{\bar{x}},\ x \partial_{\bar{z}}+\partial_{\bar{y}},\ \partial_{\bar{z}} \}
\]
by
\[
\phi(\partial_t)=0, \qquad
\phi(\partial_x)=i\partial_x, \qquad
\phi(\partial_z)=i\partial_z, \qquad
\phi(\bar{x}\partial_z+\partial_y)=i(\bar{x}\partial_z+\partial_y),
\]
and extended by complex conjugation. We claim that this structure is \(\nabla\)-Einstein with $c=0$.

The associated fundamental form \(F\) and the differential of \(\eta\) are given by
\[
F = -iu\,(dx \wedge d\bar{x} + dy \wedge d\bar{y}) - i\,\gamma \wedge \bar{\gamma},
\]
\[
d\eta = i\, (dx \wedge d\bar{x} + dy \wedge d\bar{y}).
\]
From this, it follows that \(c = \frac{1}{2} g(d\eta, F) = 0\).
Moreover, by the argument at the beginning of the example, \(\rho^B = 0\).

Therefore, to conclude that the structure is \(\nabla\)-Einstein, it remains to verify that the torsion is closed, namely
\[
d\eta \wedge d\eta + dd^\phi F = 0.
\]
Using \(dd^\phi F = 2i\,\partial\bar{\partial}F\), a straightforward computation gives
\[
d\eta \wedge d\eta + 2i\, \partial\bar{\partial}F
= (2 + u_{x\bar{x}} + u_{y\bar{y}})\, dx \wedge d\bar{x} \wedge dy \wedge d\bar{y}.
\]
Since \(u_{x\bar{x}} + u_{y\bar{y}} = -2\), the claim follows. \par
Also this example, as the previous one, is non-flat and with non-parallel torsion. For instance, these could be easily checked by observing that $d\eta$ has non-constant norm. 

\end{ex}

Observe that if we cross the two examples above with $\R$, the resulting products inherit induced non-compact Bismut--Hermite--Einstein structures (see Remark \ref{rmk:oddeven}).
\medskip
\subsection{Compact \texorpdfstring{$\nabla$}{}-Einstein manifolds in dimension 5}
In the compact case, it was recently proved in \cite{KS} that for a compact $\nabla$-Einstein manifold $(M,\phi,\xi,g)$, there exists a unique normalized function $f$ such that the vector field $V := \theta^\sharp - \mathrm{grad} f$ is $\nabla$-parallel, where we recall that $\theta$ is identified with the transverse Lee form. When non-zero, this vector field $V$ lies in $\mathcal{F}^\perp_\xi$ \cite{KS}.
 
\begin{prop} \label{prop:Vneq0}
Let $(M,\phi,\xi,g)$ be a compact $5$-dimensional $\nabla$-Einstein manifold. If $V \neq 0$ then $\nabla$ is flat, and so $(M,\phi,\xi,g)$ is constructed as in Theorem \ref{thm:flat}. 
\end{prop}
\begin{proof}
We have already observed that, for any $X,Y$, 
\[
R^\nabla(X,Y) \in \Lambda^2(T^*M) \otimes \Lambda^{1,1}_\phi(\mathcal{F}^\perp_\xi).
\] 
Furthermore, since $V$ and $\phi V$  are parallel, they lie in the kernel of $R^\nabla(X,Y)$, i.e.,
\[
V, \phi V \in \ker \big(R^\nabla(X,Y)\big).
\] 
It follows that the only potentially non-zero component of $R^\nabla(X,Y)$ is 
\[
R^\nabla(X,Y)(\alpha, \phi \alpha),
\] 
where $\alpha$ and $\phi \alpha$ are chosen such that $\{\alpha, \phi \alpha, V, \phi V\}$ forms a basis of $\mathcal{F}^\perp_\xi$.  
But we have
\[
R^\nabla(X,Y)(\alpha, \phi \alpha) = R^\nabla(X,Y,\alpha, \phi \alpha) = \rho(X,Y) = 0,
\]
which concludes the proof.

\end{proof}
We now turn our attention to the case $V = 0$. In this case, $\theta = df$, so the transverse geometry is in general conformally balanced. In dimension $4$, the balanced and K\"ahler conditions coincide, so the transverse geometry is  conformally K\"ahler.

Before proving the main Theorem, we exhibit the following two remarks.
\begin{rmk} \label{rmk:fbasic}
Let $(M,\phi,\xi,g)$ be a compact $5$-dimensional $\nabla$-Einstein manifold. If $V = 0$, and hence $\theta=df$, then $\mathcal{L}_\xi f=0$. \par
First note that $\delta^{T} \theta=\delta\theta$, where $\delta^{T}$ is the codifferential with respect to $g^T$: this can be easily proven by taking a $\phi$-basis, in fact
\[
\delta\theta=-\nabla^{\mathrm{LC}}_{\xi} \theta (\xi)- \nabla^{\mathrm{LC}}_{x_{i}}\theta (x_i)- \nabla^{\mathrm{LC}}_{\phi x_{i}}\theta (\phi x_i)=- \nabla^{\mathrm{LC}}_{x_{i}}\theta (x_i)- \nabla^{\mathrm{LC}}_{\phi x_{i}}\theta (\phi x_i)=\delta^{T}\theta,
\]
where the last equality follows since $\theta(\xi)=0$, as $\theta$ is the Lee form of the transverse Hermitian structure, and $\nabla^{\mathrm{LC}}_\xi \xi=0$. 
Now, using that $\xi$ is Killing we get
\[
0=\mathcal{L}_\xi g (dd^\phi F, F \wedge F)=\mathcal{L}_\xi (g^T (dd^\phi F, F \wedge F))=2  \mathcal{L}_\xi (\delta^T \theta)= 2 \mathcal{L}_\xi (\delta \theta)=-2 \mathcal{L}_\xi (\Delta f)=-2 \Delta (\mathcal{L}_\xi f).
\]
Since $M$ is compact, this gives $\mathcal{L}_\xi f=const$, which in turn implies that $\mathcal{L}_\xi f=0$, again by the compactness of $M$.
\end{rmk}
\begin{rmk} \label{rmk:c=0}
    Observe that $c=0$ implies
\[
\rho^\nabla=\rho^B=0,
\]
so the transverse geometry is \emph{Calabi Yau with torsion} and $d\eta$ is primitive. Then
\[
\int_{M} \frac{|d\eta|^2}{2} vol=\int_{M} \frac{|d\eta|^2}{2} \eta \wedge vol^T=-\int_{M} \eta \wedge d\eta \wedge d\eta= \int_{M} \eta \wedge dd^\phi F= \int_{M} d(\eta \wedge d^\phi F)=0,
\]
implying that $d \eta=0$. In this case, $M$ is locally $\R \times N$ where $N$ is a $4$-dimensional Bismut-Hermite--Einstein manifold. Since $M$ admits a nowhere vanishing closed $1$-form $\eta$, by Tishler Theorem \cite{Tis}, $M$ has the structure of a mapping torus of a compact $4$-dimensional Bismut-Hermite--Einstein manifold. 
Therefore, in the following we will always assume $c \neq 0$. Observe that the splitting above is no longer true in the non-compact case when $c=0$ (see Example \ref{ex:noncompactBHE}). \medskip

\end{rmk}
\begin{thm} \label{thm:caracterization}
Let $(M,\phi,\xi,g)$ be a compact $5$-dimensional $\nabla$-Einstein manifold satisfying $c \neq 0$. If $V = 0$, and hence $\theta=df$, then the transverse metric $g^T$ is conformal to a K\"ahler metric $(\tilde g^T, \phi|_{\mathcal{F}^\perp_\xi})$ satisfying:
\begin{enumerate}
    \item The scalar curvature of $\tilde g^T$ is positive and given by
    \[
    \tilde{s}^T = 2 c^2 \cdot e^f > 0.
    \]
    \item The Ricci form $\tilde \rho^T$ of $(\tilde g^T, \phi|_{\mathcal{F}^\perp_\xi})$ satisfies
    \begin{equation} \label{eqn:4dim}
        \tilde \rho^T \wedge \tilde \rho^T + \frac{1}{2} dd^\phi (\tilde s^T) \wedge \tilde \omega = 0.
    \end{equation}
\end{enumerate}

Conversely, let $(N, \tilde g^T, J, \tilde \omega^T)$ be a $4$-dimensional K\"ahler orbifold such that:
\begin{enumerate}
    \item The scalar curvature $\tilde{s}^T$ is strictly positive.
    \item The Ricci form $\tilde \rho^T$ satisfies
    \begin{equation} \label{eqn:dim41}
        \tilde \rho^T \wedge \tilde \rho^T + \frac{1}{2} dd^J (\tilde s^T) \wedge \tilde \omega = 0.
    \end{equation}
\end{enumerate}
Suppose that $M$ is a smooth $S^1$-orbibundle over $N$ 
with connection $1$-form $\eta$ such that
\[
d\eta = c^{-1} \tilde \rho^T
\]
for some non--zero constant $c \in \mathbb{R}^*$. Then $M$ admits a $\nabla$-Einstein structure.
\end{thm}
\begin{proof}

By Theorem~\ref{thm:BHE}, the $\nabla$-Hermite--Einstein condition is equivalent to
\[
c = \frac{1}{2} g(d\eta, F) = \text{const}, \quad
c\, d\eta = \rho^B, \quad dH = d\eta \wedge d\eta + dd^\phi F = 0,
\]
where $\rho^B$ denotes the Bismut Ricci form of the transverse Hermitian structure.

\medskip
We prove the first claim. Since $\theta = df$, the conformally rescaled transverse metric
\[
\tilde g^T = e^{-f} g^T
\]
is K\"ahler. In dimension $4$, the Bismut Ricci form is conformally invariant, so
\begin{equation} \label{eqn:confinv}
\tilde \rho^T = \tilde \rho^B = \rho^B = c \, d\eta.
\end{equation}

Let $\{\xi, x_i, \phi x_i\}$ be a $\phi$-basis, so that $\{ e^{f/2} x_i, e^{f/2} \phi x_i \}$ forms a local horizontal orthonormal basis for $\tilde g^T$. Taking the trace of \eqref{eqn:confinv}, we obtain
\[
\tilde s^T = \sum_{i=1}^{4} e^f \tilde \rho^T(\phi x_i, x_i) = c e^f \sum_{i=1}^{4} d\eta(\phi x_i, x_i) = 2 c^2 e^f > 0.
\]

Moreover, Equation~\eqref{eqn:4dim} follows from
\[
d\eta \wedge d\eta = \frac{1}{c^2} \tilde \rho^T \wedge \tilde \rho^T, \quad
dd^\phi F = dd^\phi(e^f \tilde \omega) = dd^\phi\left(\frac{\tilde s^T}{2 c^2} \tilde \omega\right) = \frac{1}{2c^2} dd^\phi(\tilde s^T) \wedge \tilde \omega.
\]

\medskip
We now prove that the construction is reversible.   
Assume that $(N, \tilde g^T, J, \tilde \omega^T)$ satisfies the stated hypotheses, and let $M$ be a $S^1$-orbibundle over $N$ with curvature $d\eta=c^{-1} \tilde{\rho}^T$. The total space $M$ then carries a natural normal almost contact metric structure
\[
(\phi, \eta, \xi, \tilde g = \eta^2 + \tilde g^T),
\]
and by construction $\mathcal{L}_\xi \tilde g^T = 0$.

Consider the metric
\[
g = \eta^2 + \frac{\tilde s^T}{2 c^2}\,\tilde g^T,
\qquad
g^T := \frac{\tilde s^T}{2c^2}\,\tilde g^T.
\]
Then $(\phi, \eta, \xi, g = \eta^2 + g^T)$ defines again a normal almost contact metric structure on $M$. Moreover, since $\mathcal{L}_\xi g = 0$, the connection $\nabla$ is well defined.

Let $F$ denote the fundamental form of $(\phi, \eta, \xi, g)$. By construction,
\[
F = \frac{\tilde s^T}{2c^2}\,\tilde \omega^T.
\]
We denote by $H$ the torsion of $\nabla$, which is given by
\[
H = \eta \wedge d\eta + d^\phi F=\eta \wedge d\eta+ d^\phi \left(\frac{\tilde s^T}{2c^2}\right) \wedge \tilde \omega^T.
\]

Observe that the transverse Hermitian structure $(g^T, \phi_{|\mathcal{F}^\perp_\xi} = J)$ is conformal to $(\tilde g^T, \phi_{|\mathcal{F}^\perp_\xi} = J)$. Hence, by conformal invariance of the Bismut Ricci form in real dimension $4$, it follows that
\[
\tilde \rho^T = \rho^B.
\]

Let $\{e_i, J e_i\}$ be a local orthonormal frame with respect to $\tilde g^T$. Then a local horizontal orthonormal frame for $g^T$ is given by
\[
\left\{\sqrt{\frac{2c^2}{\tilde s^T}} \, e_i,\ \sqrt{\frac{2c^2}{\tilde s^T}} \, J e_i \right\}
=
\left\{\sqrt{\frac{2c^2}{\tilde s^T}} \, e_i,\ \sqrt{\frac{2c^2}{\tilde s^T}} \, \phi e_i \right\}.
\]
With respect to this frame, we compute
\[
\frac{1}{2} g(d\eta, F)
= \frac{1}{2} \frac{2c^2}{\tilde s^T} \sum_{i=1}^{2n} d\eta(\phi e_i, e_i)
= \frac{1}{2} \frac{2c^2}{\tilde s^T} \cdot \frac{1}{c} \sum_{i=1}^{2n} \tilde \rho(J e_i, e_i)
= c.
\]
Therefore, by Theorem~\ref{thm:BHE}, we obtain $\rho^\nabla = 0$. In fact, the constant $c$ satisfies
\[
c = \frac{1}{2} g(d\eta, F), \qquad c\, d\eta = \tilde \rho^T = \rho^B.
\]

Finally, we compute
\[
dH = d\eta \wedge d\eta + dd^\phi F
= \frac{1}{c^2} \tilde \rho^T \wedge \tilde \rho^T
+ dd^\phi \left(\frac{\tilde s^T}{2c^2} \tilde \omega^T \right)
= \frac{1}{c^2} \tilde \rho^T \wedge \tilde \rho^T
+ \frac{1}{2c^2} dd^J (\tilde s^T) \wedge \tilde \omega^T = 0,
\]
where the last equality follows from the hypothesis (Equation~\eqref{eqn:dim41}). Hence, $M$ admits a $\nabla$-Hermite--Einstein structure.

\end{proof}

\begin{rmk}
   Observe that when $V = 0$ and $c \neq 0$, the resulting transverse geometry is identical to that of compact Bismut-Hermitian Einstein manifolds in dimension~6, as proven in~\cite{ABLS}.

\end{rmk}
If $(M,\phi,\xi,g)$ is a compact $5$-dimensional $\nabla$-Einstein manifold satisfying $c \neq 0$ and $V = 0$, then the transverse K\"ahler metric $\tilde g^T$ satisfies equation~\eqref{eqn:dim41}, which is equivalent to the nonlinear partial differential equation
\begin{equation} \label{eqn:box}
\square \tilde s^T = \frac{(\tilde {s}^T)^{2}}{2} - |\tilde{\mathrm{Ric}}|^2,
\end{equation}
where $\square\tilde s^T = -\Delta \tilde s^T$. This equation, commonly referred to in the physics literature as the \emph{Box equation}, first appeared in~\cite{GK}.

In~\cite[Section~4.2]{CGMS}, the authors construct a family of $2$-dimensional orthotoric K\"ahler orbifold surfaces $S^{a,b,c}$, parametrized by three positive coprime integers $a,b,c$, whose K\"ahler metrics satisfy equation~\eqref{eqn:box} and have positive scalar curvature. Starting from these orbifolds, they define a family of $5$-dimensional manifolds $L^{a,b,c}$ as $S^1$-orbibundles over $S^{a,b,c}$ with curvature form
\[
d\eta = \frac{1}{2}\,\tilde{\rho}^T,
\]
where $\tilde{\rho}^T$ denotes the Ricci form of the K\"ahler base (see also~\cite{ALL} for further details on the geometry of $L^{a,b,c}$). These manifolds were originally introduced in~\cite{CGMS} in the context of supersymmetric $\mathrm{AdS}_3 \times Y_7$ solutions of type IIB supergravity. Nevertheless, by Theorem~\ref{thm:caracterization}, the manifolds $L^{a,b,c}$ provide examples of $\nabla$-Einstein manifolds with $c = 2$. \par
The direct product $S^1 \times L^{a,b,c}$ carries an induced Bismut--Hermite-Einstein structure (see Remark \ref{rmk:oddeven}), which coincides with the example recently constructed in \cite{ALL}.
\medskip

We recall the following definition \cite{FI2}.

\begin{defn}
Let $(M,\phi,\xi,g)$ be an almost contact metric manifold. The almost contact metric structure
\[
\tilde \phi = \phi, \quad 
\tilde \xi = \xi, \quad 
\tilde \eta = \eta, \quad 
\tilde g = (1-e^{-f})\,\eta^2 + e^{-f} g
\]
is called \emph{special conformal} if $\mathcal{L}_\xi f = 0$.
\end{defn}

It is worth noting that if $(\phi,\xi,g)$ is special conformal to $(\tilde \phi,\tilde \xi,\tilde g)$ and $\mathcal{L}_\xi g = 0$, then $\mathcal{L}_\xi \tilde g = 0$ and $N_{\tilde \phi} = N_\phi$ \cite{FI2}.

\medskip
Let $(M,\phi,\xi,g)$ be a compact $5$-dimensional $\nabla$-Einstein manifold with $c \neq 0$. If $V = 0$, then by the previous theorem and Remark \ref{rmk:fbasic}, $(M,\phi,\xi,g)$ is special conformal to the structure
\[
\tilde \phi = \phi, \quad 
\tilde \xi = \xi, \quad 
\tilde \eta = \eta, \quad 
\tilde g = (1-e^{-f})\,\eta^2 + e^{-f} g,
\]
where $f$ is the unique normalized smooth function satisfying $df = \theta$, with $\theta$ denoting the Lee form of the transverse Hermitian structure. 

By the considerations above, $(\tilde \phi,\tilde \xi,\tilde g)$ is again a normal almost contact metric structure and satisfies $\mathcal{L}_\xi \tilde g = 0$. In particular, the Friedrich--Ivanov connection $\tilde \nabla$ is still well defined. \par
Since $\tilde F = e^{-f} F$ and $\theta = df$, 
\[
d\tilde F = 0.
\]
This has two consequences. First, the torsion $\tilde H$ of the connection $\tilde \nabla$ is given by
\[
\tilde H = \tilde \eta \wedge d\tilde \eta = \frac{1}{c} \eta \wedge \tilde \rho^T,
\]
where $\tilde \rho^T$ denotes the Ricci form of the transverse K\"ahler structure. Secondly, the condition $d\tilde F = 0$ implies that $(\tilde \phi,\tilde \xi,\tilde g)$ is quasi-Sasaki. Therefore, we obtain the following result.

\begin{cor} \label{corolnablaHermEins}
 Let $(M,\phi,\xi,g, \nabla)$ be a compact $5$-dimensional $\nabla$-Einstein manifold with $c \neq 0$. If $V = 0$, and hence $\theta = df$, then $(\phi,\xi,g)$ is special conformal to a quasi-Sasaki structure $(\tilde \phi,\tilde \xi,\tilde g)$.
\end{cor}
\medskip

The previous discussion enables us to establish a classification theorem for compact $5$-dimensional $\nabla$-Einstein manifolds.
\begin{thm} \label{thm:5}
Let $(M,\phi,\xi,g,\nabla)$ be a compact $5$-dimensional $\nabla$-Hermite--Einstein manifold. Let $f$ be the unique normalized smooth function such that the vector field
\[
V = \theta^\sharp - \mathrm{grad} f
\]
is parallel. Then the following hold.

\begin{enumerate}
\item If $V \neq 0$, then $(M,\phi,\xi,g,\nabla)$ is $\nabla$-flat (see Theorem~\ref{thm:flat}). In particular, its universal cover is isometric to $\R \times (\R \times \mathrm{SU}(2))$. 

\item If $V = 0$, the following cases occur:
\begin{enumerate}
    \item If $c = 0$, then $M$ is a mapping torus of a compact K\"ahler Ricci-flat manifold.
    
    \item If $c \neq 0$, then:
    \begin{enumerate}
        \item If $f$ is constant, then $\nabla$ is flat. Moreover, its universal cover is isomorphic to $\mathrm{SU}(2) \times \C$, and the lifted almost contact metric structure induces the standard left-invariant Sasaki structure on the $\mathrm{SU}(2)$ factor.
        
        \item If $f$ is non-constant, then $\nabla$ is non-flat and $(M,\phi,\xi,g,\nabla)$ is locally an $S^1$-bundle over a $4$-dimensional K\"ahler manifold with strictly positive non constant scalar curvature $\tilde s^T$, whose curvature satisfies
        \[
        \square \tilde s^T = \frac{(\tilde {s}^T)^{2}}{2} - |\tilde{\mathrm{Ric}}|^2,
        \]
        where $\square \tilde s^T = -\Delta \tilde s^T$.
    \end{enumerate}
\end{enumerate}
\end{enumerate}
\end{thm}
\begin{proof}
Let us assume that $V \neq 0$. Then, by Proposition~\ref{prop:Vneq0}, the connection $\nabla$ is flat. Hence, by Theorem~\ref{thm:flat}, the universal cover is isomorphic either to $\mathbb{R} \times (\mathbb{R} \times \mathrm{SU}(2))$ or to $\mathbb{R}^5$. We claim that the latter case cannot occur.  \par

Since $\nabla$ is flat, the Lee form $\theta$ is $\nabla$-parallel. In particular,
\[
\nabla V = -\nabla \mathrm{grad} f = 0,
\]
so that $f$ is constant. As $V \neq 0$, it follows that $V^\flat = \theta$ is non-zero.

On the other hand, since the Sasaki with torsion structure lifts to a $\nabla$-flat left-invariant Sasaki with torsion structure on $\R^5$, the Lee form vanishes. This would imply $V = 0$, a contradiction.  \medskip

We now consider the case $V = 0$. If $c = 0$, then by Remark~\ref{rmk:c=0}, $M$ is the mapping torus of either a K\"ahler Ricci-flat manifold or $\mathrm{SU}(2) \times S^1$ endowed with its standard Bismut-flat structure. We claim that the latter case cannot occur.

Indeed, since $V=0$, $\theta = df$. On $\mathrm{SU}(2) \times S^1$, the Lee form $\theta$ has constant non-zero norm, since it is invariant and the Bismut flat structure is non K\"ahler. Therefore, $|df|^2$ is constant and non-zero. 
However, by the compactness of $M$, there exists at least one point $p$ in which $df(p)=0$, implying that $df=0$ everywhere. The contradiction follows. 
\medskip

Assume now that $c \neq 0$. We distinguish two cases, depending on whether $f$ is constant or not.

\smallskip

\noindent
Assume that $f$ is constant. Since $V=0$, in this case $\theta = 0$, and hence the transverse geometry is K\"ahler. The defining equations reduce to
\[
c\, d\eta = \rho^{\mathrm{LC}}, \qquad d\eta \wedge d\eta = 0, \qquad c = \mathrm{const}.
\]

Write
\[
d\eta = \frac{c}{2} F + \gamma,
\]
where $\gamma$ is a horizontal anti-self-dual $2$-form. The condition $d\eta \wedge d\eta = 0$ implies that $|\gamma|^2 = c^2$, and so $d\eta$ has constant norm. Moreover, since $c \neq 0$, it follows that the kernel of $d\eta$ restricted to the horizontal distribution has constant dimension equal to $2$.

We therefore obtain the orthogonal splitting
\[
TM = \mathcal{F}_\xi \oplus \mathcal{K} \oplus \mathcal{J},
\]
where $\mathcal{K} := \ker(d\eta)_{\mathcal{H}}$ is the kernel of $d\eta$ on the horizontal distribution, and $\mathcal{J}$ is the orthogonal complement of $\ker(d\eta)_{\mathcal{H}}$ with respect to the transverse metric. Observe that, with respect to $g$,
\[
\mathcal{K}^\perp = \mathcal{F}_\xi \oplus \mathcal{J}.
\]

Since $d\eta$ is of type $(1,1)$ on horizontal vectors, both $\mathcal{K}$ and $\mathcal{J}$ are $\phi$-invariant.

Using $c\, d\eta = \rho^{\mathrm{LC}}$ and the fact that the transverse scalar curvature is constant and strictly positive (see Theorem~\ref{thm:caracterization}), it follows that the transverse Ricci tensor of the K\"ahler metric has two distinct non-negative constant eigenvalues, one of which is zero.
\medskip

In this setting, one can apply the argument of Corollary~1.2 in \cite{ABLS} (based on a transverse version of Theorem~1 in \cite{ADM}) to conclude that the transverse Ricci tensor, and hence $\rho^{\mathrm{LC}}$, is parallel with respect to the transverse Levi-Civita connection. Since $d\eta$ is a constant multiple of $\rho^{\mathrm{LC}}$, it is transversely parallel, and hence so is $\mathcal{K}$.

\medskip

We claim that $d\eta$ is parallel with respect to $\nabla$. Let us denote by $\nabla^T$ the transverse Levi-Civita connection. A direct computation, using Equation \eqref{eqn:xieta} and item 4. in Remark \ref{rmk:rmk}, shows that
\[
(\nabla_x d\eta)(y,z) = (\nabla^T_x d\eta)(y,z) = 0,
\]
\[
(\nabla_\xi d\eta)(X,Y) = \mathcal{L}_\xi d\eta (X,Y) + g(\iota_X d\eta, \iota_Y d\eta) - g(\iota_X d\eta, \iota_Y d\eta) = 0, 
\]
\[
(\nabla_X d\eta)(\xi,Y) = 0,
\]
for all horizontal vector fields $x,y,z$ and generic vector fields $X,Y$. 

\medskip

Using this fact, we claim that $\mathcal{K}$ is also parallel with respect to the Levi-Civita connection of $g$.

Let $x \in \mathcal{K}$ and let $y$ be a horizontal vector field. Denote by $\nabla^T$ the transverse Levi-Civita connection. Then
\[
\nabla^{\mathrm{LC}}_y x = \nabla^T_y x \in \mathcal{K},
\]
since $\mathcal{K}$ is transversely parallel.

It remains to show that
\[
\nabla^{\mathrm{LC}}_\xi x \in \mathcal{K}.
\]
To this end, let $Y$ be a generic vector field. Using the relation between $\nabla^{\mathrm{LC}}$ and $\nabla$, we compute
\[
d\eta(\nabla^{\mathrm{LC}}_\xi x, Y)
= d\eta(\nabla_\xi x, Y) - \frac{1}{2} d\eta(\iota_x d\eta, Y).
\]
Since $x \in \mathcal{K}$, the second term vanishes, and by the $\nabla$-parallelism of $d\eta$ we obtain
\[
d\eta(\nabla_\xi x, Y) = \nabla_\xi d\eta \ (x,Y) = 0.
\]
Hence $d\eta(\nabla^{\mathrm{LC}}_\xi x, Y) = 0$ for all $Y$, and therefore
\[
\nabla^{\mathrm{LC}}_\xi x \in \mathcal{K}.
\]

We conclude that $\mathcal{K}$ is parallel. Consequently, its orthogonal complement
\[
\mathcal{F}_\xi \oplus \mathcal{J}
\]
is also parallel. By the de Rham decomposition theorem, the universal cover of $M$ splits as a Riemannian product
\[
\widetilde{M} \cong N^3 \times \mathbb{C},
\]
where $TN^3 = \mathcal{F}_\xi \oplus \mathcal{J}$.

Since the induced structure is $\nabla$-Einstein, the same holds on $N^3$. In dimension $3$, this implies flatness, hence $N^3$ is isomorphic to $\mathrm{SU}(2)$, by Theorem~\ref{thm:flat}. Up to isomorphism, this is the standard left-invariant Sasaki structure on $\mathrm{SU}(2)$.
\noindent

\medskip 
\noindent

If $f$ is non-constant, then $\nabla$ is non-flat. Indeed, otherwise $\theta$ would have constant norm (since the structure on the universal cover is left-invariant), and hence $df = \theta$ would also have constant norm, yielding a contradiction. The remaining claims follow by Theorem \ref{thm:caracterization}.

\end{proof}

\begin{rmk}
The assumptions of Theorem \ref{thm:5} can be slightly relaxed. Indeed, as recalled in the preliminaries, in dimension five normality is equivalent to the skew-symmetry of the Nijenhuis tensor \cite{CM}.
\end{rmk}
A general result on the properties of compact $\nabla$-Einstein manifolds with $V \neq 0$ was recently established in \cite[Theorem 5.6]{KS}.
\subsection{7-dimensional analysis}
In this section we collect some results for compact $\nabla$-Einstein manifolds in dimension $7$. 
\begin{thm}
\label{thm:7dimensional-structure}
Let $(M^7,\phi,\xi,\eta,g,\nabla)$ be a compact $\nabla$-Einstein
Sasaki-with-torsion manifold and consider the vector field
\[
V:=\theta^\sharp-\operatorname{grad}(f),
\]
where $\theta$ is the transverse Lee form. 
Then the following statements hold.

\begin{enumerate}
\item Suppose that $V\neq0$. Set
\[
\mu:=V^\flat,
\qquad
\nu:=(\phi V)^\flat=\phi\mu,
\]
and consider the orthogonal decomposition
\[
TM=\mathcal G\oplus\mathcal G^\perp,
\qquad
\mathcal G:=\langle V,\phi V,\xi\rangle.
\]
Then $\mathcal G$ is integrable and defines a $3$-dimensional foliation
with flat leaves, locally isometric to $\mathbb R^3$. Moreover, $M$ is
locally an $\mathbb R^3$-bundle over a $4$-dimensional Hermitian
manifold with Hermitian form $\omega^T$ and complex structure induced
by $\phi$. The forms $d\mu$, $d\nu$, and $d\eta$ are basic real $(1,1)$-forms on
the local quotient, and the transverse Hermitian structure is conformally K\"ahler. More
precisely,
\[
\widetilde\omega^T=e^{-f}\omega^T
\]
is K\"ahler. If $\widetilde\rho^T$ and $\widetilde s^T$ denote,
respectively, its Ricci form and scalar curvature, then $c$ is constant
and
\begin{equation}
\label{eqn:rho2}
\widetilde\rho^T=d\nu+c\,d\eta,
\qquad
\widetilde s^T=2e^f(c^2+1)>0.
\end{equation}
Finally, the strong condition $dH=0$ is equivalent to
\begin{equation}
\label{eqn:dim7}
\begin{split}
0={}&
d\mu\wedge d\mu
+
(\widetilde\rho^T-c\,d\eta)
 \wedge
(\widetilde\rho^T-c\,d\eta)
+
d\eta\wedge d\eta
+
dd^\phi\left(
\frac{\widetilde s^T}{2(c^2+1)}
\right)
\wedge\widetilde\omega^T.
\end{split}
\end{equation}

\item Suppose that $V=0$. Then $\theta=df$, and the transverse
$6$-dimensional Hermitian geometry is conformally balanced. Locally,
$M$ is an $S^1$-bundle over a conformally balanced Hermitian manifold
$N^6$, and the $\nabla$-Einstein equations are equivalent to
\begin{equation}
\label{eqn:7dim}
\rho^B=c\,d\eta,
\qquad
b=c^2=\mathrm{const},
\qquad
d\eta\wedge d\eta+dd^\phi F=0,
\qquad
\theta=df,
\end{equation}
where $\rho^B$ and $b$ denote, respectively, the Bismut Ricci form and
the Bismut scalar curvature of $N$.

More precisely:

\begin{enumerate}
\item If $c=0$, then
\[
\rho^B=0,
\qquad
g(d\eta,F)=0,
\qquad
d\eta\wedge d\eta+dd^\phi F=0.
\]
Thus, locally, $M$ is an $S^1$-bundle over a conformally balanced
Calabi--Yau-with-torsion manifold $N^6$, whose curvature
$d\eta\in\Omega^{1,1}(N)$ satisfies
\[
g(d\eta,F)=0,
\qquad
d\eta\wedge d\eta+dd^\phi F=0.
\]

\item If $c\neq0$, then $b=c^2>0$ and
\[
d\eta=\frac1c\rho^B,
\qquad
\frac1b\,\rho^B\wedge\rho^B+dd^\phi F=0.
\]
Thus, locally, $M$ is an $S^1$-bundle over a conformally balanced
Hermitian manifold $N^6$ with constant positive Bismut scalar
curvature.
\end{enumerate}
\end{enumerate}
\end{thm}

\begin{proof}
Assume that $V=\theta^\sharp-\mathrm{grad}(f) \neq 0$ and denote $\mu := V^{\flat}$, $\nu := \phi V^{\flat}$. We decompose the tangent bundle as $TM = \langle V, \phi V, \xi \rangle \oplus \langle V, \phi V, \xi \rangle^{\perp}$, the first summand being denoted by $\mathcal{G}$. By \cite[Theorem 5.6]{KS}, the torsion can be written as
\begin{equation} \label{eqn:H}
  H = \mu \wedge d\mu + \nu \wedge d\nu + d^{\phi}\omega^{T} + \eta \wedge d\eta,  
\end{equation}
where $\omega^{T}$ is the transverse Hermitian form on $\mathcal{G}^\perp$ and $d^{\phi}\omega^{T} + \eta \wedge d\eta$ is basic with respect to $V$ and $\phi V$. Using that $d^{\phi}\omega^{T} + \eta \wedge d\eta$ is basic, $V, \phi V$ and $\xi$ are mutually orthogonal \cite[Theorem 5.6]{KS} and $\phi\xi = 0$, we have
\[
0 = \iota_{\xi}\iota_{V}(d^{\phi}\omega^{T} + \eta \wedge d\eta) = -\iota_{V} d\eta,
\]
which implies $\mathcal{L}_{V}\eta = 0$. Exploiting that $d\eta$ is of type $(1,1)$, we also obtain $\mathcal{L}_{\phi V}\eta = 0$. Also, using that $V$ and $\phi V$ are $\nabla$-parallel, $\iota_V H=d\mu$ and $\iota_{\phi V} H=d\nu$, and therefore $\iota_{V} d\mu = 0$ and $\iota_{\phi V} d\nu = 0$. Moreover, by Equation \eqref{eqn:H}, 
\[\iota_{\phi V} d\mu = 0,\  \ \  \iota_{V} d\nu = 0.\]
This implies, for instance, that $\mathcal{L}_V \nu=0$ and so $[V,\phi V]=0$ since $V$ is Killing. 
Therefore the distribution $\langle V, \phi V, \xi \rangle$ defines a $3$-dimensional foliation locally isometric to $\mathbb{R}^{3}$, and the following relations hold:
\[
\iota_{\xi} d\mu = 0,\,\, \iota_{\xi} d\nu = 0.
\]
Furthermore, using $F = \nu \wedge \mu + \omega^{T}$ and the expression of the torsion $H=\eta \wedge d\eta+ d^\phi F$, the identity
\[
 \mu \wedge \phi d\mu + \nu \wedge \phi d\nu + d^{\phi}\omega^{T}= \phi dF = \mu \wedge d\mu + \nu \wedge d\nu + d^{\phi}\omega^{T},
\]
immediately implies that $d\mu$, $d\nu$ and $d\eta$ belong to $\Omega^{(1,1)}(\mathcal{G}^{\perp})$.\par
We claim that the transverse Hermitian structure is basic. Indeed, using that
\[
0 = \iota_{\phi V} d^{\phi}\omega^{T} = -d\omega^{T}(V, \phi X, \phi Z)
\]
for all vector fields $X,Z$ and $\phi^{2} = -\mathrm{Id}$ on $\mathcal{G}^\perp$, this implies $\iota_{V} d\omega^{T} = 0$ and consequently $\mathcal{L}_{V}\omega^{T} = 0$. The same argument applied to $\phi V$ yields $\mathcal{L}_{\phi V}\omega^{T} = 0$. Hence we obtain the local structure of an $\mathbb{R}^{3}$-bundle over a $4$-dimensional Hermitian manifold $(\omega^{T}, \phi)$.\par
Denote by $\rho^B$ the Bismut Ricci form of $(\omega^{T}, \phi)$. Using a generalization of Theorem \ref{thm:BHE} to local $\R^{2k+1}$-bundles,
\[
\rho^\nabla=\rho^B-d(c_\mu \mu)-d(c_\nu \nu)-d(c \eta),
\]
where $c_\mu:= \frac{1}{2} g(d\mu, \omega^T)$ and $c_\nu:= \frac{1}{2} g(d\nu, \omega^T)$. By \cite[Theorem 5.6]{KS}, $c_\mu=0$ and $c_\nu=1$\footnote{Note that there is sign and normalizations discrepancy, due to different conventions.}. Therefore the $\nabla$-Einstein condition imposes that $c$ is constant and
\[
\rho^B=d \nu +c d\eta.
\]
Furthermore, again by \cite[Theorem 5.6]{KS} the transverse Lee form of $(\omega^{T}, \phi)$ satisfies $\theta^T=df$, and so $(\omega^{T}, \phi)$ is conformally K\"ahler. Denoted by $\tilde \omega^T$ the conformal K\"ahler form,  and using again the conformal invariance of $\rho^B$ in dimension $4$, we get that the vanishing of $\rho^\nabla$ is equivalent to
\begin{equation} \label{eqn:rho2}
    \tilde \rho^{T}=d \nu +c d\eta,
\end{equation}
where $\tilde \rho^{T}$ denotes the Ricci curvature of the transverse K\"ahler metric. Tracing \eqref{eqn:rho2} we get again that the transverse K\"ahler structure has positive scalar curvature given by
\[
\tilde s^T=2e^f(c^2+1).
\]
The closure of $H$ is therefore equivalent to 
\[
0=d\mu \wedge d\mu+ d\nu \wedge d\nu+d\eta \wedge d\eta+dd^c \omega^T=d\mu \wedge d\mu+ d\nu \wedge d\nu+d\eta \wedge d\eta+dd^\phi \left(\frac{\tilde s^T}{2(c^2+1)}\right)\wedge \tilde \omega^T,
\]
which can be slightly rewritten as 
\begin{equation} \label{eqn:dim7}
    0=d\mu \wedge d\mu+ (\tilde \rho^T -c d\eta) \wedge (\tilde \rho^T -c d\eta)+d\eta \wedge d\eta+dd^\phi \left(\frac{\tilde s^T}{2(c^2+1)}\right)\wedge \tilde \omega^T.
\end{equation}

\medskip

If we instead require that $V=0$, then the transverse geometry is conformally balanced, since $d\theta=d^2f=0$. Therefore, the $\nabla$-Einstein equation for $V=0$ is equivalent to the following system of equations
\begin{equation} \label{eqn:7dim}
    c d\eta=\rho^B, \ \ c^2=b=\text{const}, \ \ d\eta \wedge d\eta + dd^\phi F=0, \ \ \ \theta=df. 
\end{equation}

For $c=0$, Equation \eqref{eqn:7dim} reduces to 
\[
d\eta \wedge d\eta + dd^\phi F=0, \ \ \ \theta=df, \ \ \rho^B=0, \ \ g(d\eta,F)=0,
\]
that is, $M$ has the local structure of an $S^1$-bundle over a Calabi Yau with torsion conformally balanced $6$-dimensional manifold $N$, whose curvature $d\eta \in \Omega^{1,1}(N)$ satisfies
\[
d\eta \wedge d\eta + dd^\phi F=0, \ \ g(d\eta,F)=0.
\]
For $c \neq 0$, instead, Equation \eqref{eqn:7dim} reduces to 
\[
d\eta=\tfrac{1}{c}\rho^B, \ \ c^2=b=\text{const}, \ \  \frac{1}{c^2} \rho^B \wedge \rho^B + dd^\phi F=0, \ \ \ \theta=df, 
\]
that is, $M$ has the local structure of an $S^1$-bundle with curvature $d\eta=\tfrac{1}{\sqrt b}\rho^B$ over a conformally balanced $6$-dimensional manifold $N$ such that the Bismut scalar curvature $b$ is constant and positive, and the following equation holds
\[
 \frac{1}{b}  \rho^B \wedge \rho^B + dd^\phi F=0.
\]
\end{proof}

\begin{rmk}
 Observe that although equation \eqref{eqn:dim7} is a natural generalization of equation \eqref{eqn:4dim}, it does not reduce to a \emph{nice} scalar equation as for the $5$-dimensional case. However, when $c=0$ (which, unlike the $5$-dimensional case, is not an excluded value), then 
\[
0=dH=d\mu \wedge d\mu+ \tilde \rho^T \wedge \tilde \rho^T +d\eta \wedge d\eta+dd^\phi \left(\frac{\tilde s^T}{2}\right)\wedge \tilde \omega^T,
\]
which, up to take the Hodge star with respect to the K\"ahler metric, coincides with 
\begin{equation} \label{eqn:gen_box}
  \tilde \Delta \tilde s^T +\frac{(\tilde s^T)^2}{2}-|\tilde {\mathrm{Ric}}|^2-|d \mu |^2-|d\eta|^2=0,  
\end{equation}
which is a natural generalization of the Box equation \eqref{eqn:box}. Observe that when $d\eta=0$, examples of $7$-dimensional $\nabla$-Einstein can be constructed from the known examples in dimension $5$ and $6$. For instance, if we take $L^{a,b,c}$ and we perform a trivial $\mathbb{T}^2$ bundle with $d\mu=0$ and $d\eta=0$, we immediately get that equation \eqref{eqn:gen_box} is satisfied. Similarly, if we take the $6$-dimensional examples constructed in \cite{ALL}, the $4$-dimensional geometry satisfies 
\begin{equation} 
  \tilde \Delta \tilde s^T +\frac{(\tilde s^T)^2}{2}-|\tilde {\mathrm{Ric}}|^2-|d \mu |^2=0,  
\end{equation}
for an anti-self dual form $d\mu$. Trivial $S^1$-bundles over that examples (with $d\eta=0$), give other solutions to the generalized Box equation \eqref{eqn:gen_box}.

\end{rmk}
\section{Geometric flows}
In this section we investigate geometric flows on manifolds with almost contact metric structures. To the best of our knowledge, the literature on such geometric flows is rather scarce. The main goal of this section is to first develop a general framework for studying flows of almost contact metric structures using the formalism introduced in \cite{Fadel2022} and secondly we specialize to case of Sasaki structures with torsion leading to an analogue of the pluriclosed flow.

\subsection{General framework for flows of almost contact metric structures}

Let $(M^{2n+1},\varphi,\xi,\eta,g)$ be a general almost contact metric structure. Equivalently, this means that $M$ admits a $\mathrm{U}(n)':=1\times \mathrm{U}(n)$-structure, where it is viewed as a subgroup of $\mathrm{SO}(2n+1)$ \cite{Bl,CM}. The infinitesimal deformation space of a $\mathrm{U}(n)'$-structure is isomorphic to the representation $\mathfrak{m}$ defined by 
\begin{equation*}
\mathfrak{gl}(2n+1,\R)= \mathfrak{u}(n)'\oplus \mathfrak{m},    
\end{equation*}
i.e. $\mathfrak{m}$ is the orthogonal complement of $\mathfrak{u}(n)'\subset\mathfrak{so}(2n+1)\subset\mathfrak{gl}(2n+1,\R)$.
To describe geometric flows of $\mathrm{U}(n)'$-structures, we need a more concrete description of the space $\mathfrak{m}$, which we shall view as a subbundle of $\mathrm{End}(TM) \cong T^*M\otimes T^*M$.

We begin by first identifying the tangent space at a fixed point $x\in M$ with the $\mathrm{U}(n)'$-module $\widehat{V}:=\R \oplus V$, where $V\cong \mathbb{R}^{2n}$ denotes the standard $\mathrm{U}(n)$ representation and the trivial module is spanned by the Reeb vector field $\xi$. We always implicitly identify vectors and $1$-forms using the metric $g$.
Following \cite{Salbook}, we denote by $[\Lambda^{p,q}]$ the real vector space of $(p+q)$-forms which underlies the complexified module $\Lambda^{p,q}(V_{\mathbb{C}})\oplus \Lambda^{q,p}(V_{\mathbb{C}})$. Recall that $F\in [\Lambda^{1,1}]$ and we denote its orthogonal complement by $[\![\Lambda^{1,1}_0]\!]$. A simple calculation then shows there is a splitting into irreducible $\mathrm{U}(n)'$-modules:
\begin{gather*}
    \Lambda^2(\widehat{V}) \cong \{ \eta \wedge \alpha\ |\ \alpha\in \Lambda^1(V) \} \oplus \langle F\rangle \oplus [\![\Lambda^{1,1}_0]\!] \oplus [\Lambda^{2,0}].
\end{gather*}
Similarly, for the space of symmetric $2$-tensors, we have:
\begin{gather*}
    S^2(\widehat{V}) \cong 
    \langle \eta \otimes \eta \rangle \oplus  \{ \eta \odot \alpha\ |\ \alpha\in \Lambda^1(V) \} \oplus  \langle g_T\rangle \oplus S^2_+ \oplus S^2_-, 
\end{gather*}
where $\odot$ denotes the symmetric product, $S^2_+\subset S^2(V)$ is the space of traceless $\varphi$-invariant symmetric $2$-tensors, i.e. $\kappa(\cdot,\cdot)=\kappa(\varphi\cdot,\varphi\cdot)$ and $\mathrm{tr}_g\kappa=0$, and $S^2_-\subset S^2(V)$ is the space of $\varphi$-anti-invariant symmetric $2$-tensors, i.e. $\kappa(\cdot,\cdot)=-\kappa(\varphi\cdot,\varphi\cdot)$. It is worth noting that as $\mathrm{U}(n)'$-modules, the following spaces are isomorphic:
\begin{gather*}
    \mathfrak{su}(n)\cong S^2_+\cong [\![ \Lambda^{1,1}_0]\!],\\
    V\cong [\Lambda^{1,0}]\cong \{ \eta \wedge \alpha\ |\ \alpha\in \Lambda^1(V) \} \cong \{ \eta \odot \alpha\ |\ \alpha\in \Lambda^1(V) \}.
\end{gather*}
Thus, as a $\mathrm{U}(n)'$ representation we have:
\begin{equation*}
    \mathfrak{m} \cong 2 \mathbb{R}\oplus 2[\Lambda^{1,0}] \oplus [\![\Lambda^{1,1}_0]\!] 
      \oplus [\Lambda^{2,0}]\oplus  S^2_-  ,
\end{equation*}
which is consistent with the dimensional count: 
\[
\mathrm{dim}(\mathfrak{m})=(3n+1)(n+1)=2(1)+2(2n)+(n^2-1)+n(n-1)+n(n+1).\]
A deformation of a $\mathrm{U}(n)'$-structure is therefore determined by choosing a tensor in each irreducible summand of $\mathfrak{m}$. It will be important, however, to distinguish the symmetric and skew-symmetric components of $\mathfrak{m}\subset \mathfrak{gl}(2n+1,
\R)\cong \Lambda^2(\widehat{V})\oplus S^2(\widehat{V})
$ in the subsequent.

Next we introduce the contraction operator $\diamond$, defined by
\[
(A \diamond \alpha)_{i_1\cdots i_l}^{j_1\cdots j_k}= -\sum_{p=1}^kA_{m j_p} \alpha^{j_1\cdots m \cdots j_k}_{i_1\cdots i_l}+\sum_{q=1}^lA_{i_q m}\alpha^{j_1\cdots j_k}_{i_1\cdots m\cdots i_l },
\]
where $A=A_{ij}$ is a $(0,2)$-tensor, $\alpha=\alpha_{i_1\cdots i_l}^{j_1 \cdots j_k}$ is a $(k,l)$-tensor, and the index $m$ in the above expression is in the $j_p$ and $i_q$ position in $\alpha$. The operator $\diamond$ corresponds to the infinitesimal action of the group $\mathrm{GL}(2n+1,\R)$ on the tensor $\alpha$, see \cite{Fadel2022}.
Having defined $\mathfrak{m}$ and $\diamond$, we can now describe the general $\mathrm{U}(n)'$-flows.
\begin{prop}
Denoting by $\Phi=(\varphi,\xi,\eta,g)$ the data of the almost contact metric structure, a general $\mathrm{U}(n)'$-flow can be expressed as
\begin{equation}\label{equ: general flow A}
    \frac{\partial}{\partial t} \Phi = A \diamond \Phi,
\end{equation}
where $A$ corresponds to a section of the associated bundle to $\mathfrak{m}$. Hence, a general $\mathrm{U}(n)'$ deformation is determined by the choice of the tensor $A$ which can be expressed as
\begin{equation}
    A = \Big(f_0 \eta \otimes \eta + h_0 g_T + \eta \odot \alpha_1 + \alpha^2_+ +\alpha^2_-\Big)+ \Big(\eta \w \beta_1 +\beta^{2,0}\Big) \in  S^2(T^*M)\oplus\Lambda^2(T^*M),\label{eq: general A}
\end{equation}
where $f_0,h_0$ are functions, $\alpha_1,\beta_1$ are horizontal $1$-forms, $\alpha^2_\pm$ are symmetric $2$-tensors in $S^2_{\pm}$, and $\beta^{2,0}$ a horizontal $2$-form such that $\varphi(\beta)=-\beta$.    
\end{prop}
The reader might find it helpful to compare  with the case of $\SU(3)$-flows on $6$-manifolds considered in \cite{fadelflows2} which has a similar decomposition. Note that by construction we always have $\big(\langle F\rangle\oplus [\![\Lambda^{1,1}_0]\!]\big) \diamond \Phi=0$. For a natural $\mathrm{U}(n)'$-flow, the tensor $A$ is typically chosen among the second-order differential invariants of $\Phi$ i.e. the covariant derivative of the intrinsic torsion and the Riemann curvature tensor.

While the general $\mathrm{U}(n)'$ variation is determined by the tensor $A$, each irreducible component of $A$ may act either trivially or non-trivially on different components of the tuple $\Phi$. More precisely, we have:
\begin{lem}\label{lemma: deformation}
    Given $A$ as in \eqref{eq: general A}, the following holds:
\begin{enumerate}
    \item $\mathrm{ker}(A\diamond g) = \{\eta\wedge\beta_1, \beta^{2,0}\}$
    \item $\mathrm{ker}(A\diamond \xi) = \{\eta\otimes (f_0\eta+\alpha_1+\beta_1)\}^c$
    \item $\mathrm{ker}(A\diamond \eta) = \{ (f_0\eta+\alpha_1-\beta_1)\otimes \eta \}^c$
    \item $\mathrm{ker}(A\diamond \varphi) = \{ f_0\eta \otimes \eta ,h_0g^T,\alpha^2_+\}$
\end{enumerate}
where the superscript $^c$ denotes the complement of the components in $A$.
\end{lem}
\begin{proof}
    The result follows by a direct calculation using the definition of $A$ and $\diamond$. For instance, since $g$ is preserved by $\mathrm{SO}(2n+1)$, it follows that $\Lambda^2(T^*M)\diamond g=0$. Similarly, any $(0,2)$-tensor which contracts non-trivially with $\xi$ under the operator $\diamond$ must be of the form $\eta\otimes\kappa$ for some $1$-form $\kappa$. An analogous argument applies for $\eta$. Lastly, since $\varphi$ defines an almost complex structure on the horizontal distribution, it follows that it is invariant under $\langle g^T\rangle\oplus S^2_+\subset \mathfrak{gl}(n,\C)$. Furthermore, from the fact that $\varphi(\xi)=0$, we also know that it is invariant under $\diamond$ by  $\eta\otimes\eta$. This concludes the proof.
\end{proof}
An important application of the above lemma is to define geometric flows which only evolve part of the data of $\Phi$. For instance, if we want a $\mathrm{U}(n)'$-flow which fixes the Reeb vector field $\xi$ then it suffices to choose $A$ so that $f_0=0$ and $\alpha_1=-\beta_1$. If furthermore, we also want to fix the horizontal distribution then we have to set $\alpha_1=0$ and the flow reduces completely to the fixed transverse distribution. To the best of our knowledge, the study of flows of almost contact metric structures is rather limited in the literature; the only example of $\mathrm{U}(n)'$-flow we are aware of is the following:\vskip0.8em

\textbf{Example: The Sasaki Ricci flow.}
The Sasaki Ricci flow was introduced in \cite{Smoczyk2010} as the odd dimensional analogue to the K\"ahler Ricci flow. In this case, we have a Sasaki structure on $M$ i.e. the underlying almost contact metric structure is normal, $\xi$ is a unit  Killing vector field and $d \eta=2 F$. Equivalently, the cone metric on $M$ is K\"ahler. The Sasaki Ricci flow is then defined by: 
\begin{equation*}
    \frac{\partial}{\partial t}g^T(t)=-(\mathrm{Ric}^T_{g(t)}-\lambda g^T(t)),
\end{equation*}
where $\lambda\in\{0,\pm1\}$ is a constant chosen depending on the manifold. Thus, one evolves the transverse metric by the (translated) Ricci flow. 
In our above framework, this can be viewed as the case when $A=-\frac{1}{2}(\mathrm{Ric}^T-\lambda g^T)+ d^cf\otimes \eta $, where $f$ is a basic function (depending on $g$) chosen so that the condition $g^T(\cdot ,\varphi \cdot )=:F=\frac{1}{2}d\eta$  is preserved along the flow, see \cite{Smoczyk2010}. Intuitively, $f$ can be thought of as the deformation of the transverse K\"ahler potential.
It is easy to see from Lemma \ref{lemma: deformation} that $\xi$ is fixed under this flow while $ g, \eta$ and $\varphi$ all vary. If one considers the local quotient of $M$ by the group action generated by $\xi$, then the induced complex structure determined by $\varphi$ is invariant under the Sasaki Ricci flow - thus, this essentially corresponds to the K\"ahler Ricci flow. 

We now return to the general case of $\mathrm{U}(n)'$-flows. The intrinsic torsion tensor $T\in \Lambda^1\otimes \mathfrak{m}$  of a $\mathrm{U}(n)'$-structure is defined by the relation $\nabla_k\Phi=T_k\diamond\Phi$. It follows from the recent result in \cite{fadelflows2} that for $$A=-\mathrm{Ric}(g)+\mathrm{div}(T)+\mathrm{lower \ order\ terms},$$ 
where $\mathrm{div}(T)_{ij}:=\nabla_kT_{k,ij}$, the flow given by \eqref{equ: general flow A} exists for short time and is unique; this is the so-called Ricci harmonic flow. Moreover, this flow is also analytically well-behaved in the sense that $T$ evolves by a heat type flow. However, the flow does not in general preserve any special property of the initial data $\Phi_0$. For instance, even if $\xi(0)=\xi_0$ is a Killing vector field, $\xi(t)$ need not remain Killing. Likewise, the flow does not, in general, preserve the normality condition. We shall next describe modifications to address these two issues.

\begin{lem}\label{lemma: ricci with killing reeb}
    Let $(M^{2n+1},\varphi,\xi,\eta,g)$ be an almost contact metric structure with Killing Reeb vector field $\xi$ then its Ricci curvature is given by     \begin{equation}\label{equ: ricci s1}
        \mathrm{Ric}(g)=\frac{1}{4}g(d\eta,d\eta)\eta \otimes\eta + \frac{1}{2} \eta \odot \delta^T d\eta + \mathrm{Ric}(g^T)-\frac{1}{2} \sum_{i=1}^{2n}\iota_{e_i}d\eta \otimes\iota_{e_i}d\eta,
    \end{equation}
    where $\{e_i\}$ denotes a local orthonormal frame of the horizontal distribution and $\delta^T$ denotes the horizontal co-differential defined by $\delta^T d\eta=
    \iota_\xi  *d(\iota_\xi  * d\eta)$, where $*$ is the Hodge star operator on $M$.
\end{lem}
\begin{proof}
    This follows from O'Neill's formulae, see \cite[Proposition 9.36]{Besse}. The O'Neill tensor $T$ vanishes since the orbits are totally geodesics, while the O'Neill tensor $A$ can be identified with the curvature form $d\eta$. The proof is simply a matter of rewriting the expressions given in \cite{Besse}.
\end{proof}

Suppose we start with an almost contact metric structure $\Phi_0=(\varphi_0,\xi_0,\eta_0,g_0)$ with Killing Reeb vector field $\xi_0$. The problem is that if we define an arbitrary $\mathrm{U}(n)'$-flow then the Reeb vector field $\xi(t)$ does not remain Killing in general. On the other hand, the initial vector field $\xi_0$ will remain Killing for natural geometric flows (i.e. with $A$ defined in terms of differential invariants of $\Phi$) but its length $g(t)(\xi_0,\xi_0)$ will not stay constant along the flow in general - compare for instance with the Ricci flow on a round sphere: the length of Killing vector fields will go to zero as the sphere shrinks to a point. So we want the metric $g$ to evolve by the Ricci flow to highest order (in order to have short time existence and uniqueness) and we also want $g(t)(\xi_0,\xi_0)=1$ for all $t$. Thus, we need to consider a normalized version of Ricci flow.

Suppose that we have a general  Riemannian metric invariant under the action generated by a vector field $X$. The metric can then be expressed as:
\[
g=u^2\eta^2+g^T,
\]
where $\eta$ is the connection $1$-form for the Killing vector field $X$, i.e. $\eta:=u^{-2}g(X,\cdot)$, and $u:=g(X,X)^{\frac{1}{2}}$. Note that $u$ is a basic function since $X$ is Killing. Using O'Neill formulae, one can show that the vertical component of the Ricci curvature is given by
\begin{equation*}
\mathrm{Ric}(g)(X,X)=\frac{1}{2}\delta d(u^2)+2|\nabla u|^2+\frac{u^4}{4}g(d\eta,d\eta).
\end{equation*}
Observe that when $X$ has unit length i.e. $u=1$, we indeed recover the coefficient of the $\eta\otimes\eta$ term in the formula of Lemma \ref{lemma: ricci with killing reeb} (indeed the O'Neill tensor $T$ can be identified with $\nabla u$). The above suggests to consider the modified Ricci flow:
\begin{equation}\label{equ: modified ricci flow unit}
\frac{\partial}{\partial t} g= -2\mathrm{Ric}(g)+\frac{u^2}{2}g(d\eta,d\eta) g + \cdots,
\end{equation}
where $\cdots$ denote lower order terms not involving $\eta\otimes\eta$. Note that the added term is of lower order and as such the principal symbol of the flow is unchanged. Thus, the proof of short time existence and uniqueness follows directly from the theory of Ricci flow. In particular, the uniqueness part implies that $X$ is a symmetry of the flow.
Contracting both sides of \eqref{equ: modified ricci flow unit} with $X\otimes X$, we see that $u$ evolves by:
\[
\frac{\partial}{\partial t} (u^2)= -\delta d(u^2)-4|\nabla u|^2,
\]
where $-\delta d(u^2)=\Delta(u^2)$ is the analyst Laplacian. 
Thus, the function $u^2$ evolves by a  parabolic flow under this normalized Ricci flow. Suppose that at $t=0$ we have $u(0)=1$, then it follows by the maximum principle, or by uniqueness of solutions to parabolic equations, that $u(t)=1$. In other words, $X$ is a unit length Killing vector field for $g(t)$. Thus, this allows us to reduce the flow for $g$ to a flow for the pair $\eta$ and $g^T$.
It is worth noting that there exist other normalized versions of the Ricci flow in the literature: the most well-known probably being the volume preserving one, instead our normalization here preserves the length of $X$ if it is initially constant.

We can now apply the above to flows of $\mathrm{U}(n)'$-structures by setting $X=\xi_0$ under the hypothesis that the initial Reeb vector field is unit Killing. With the above normalized Ricci flow for the metric, we have $g(t)(\xi_0,\xi_0)=1$. We can set $\xi(t)=\xi_0$, i.e. the Reeb vector field is invariant, which implies that $f_0=0$ and $\alpha_1=-\beta_1=-\delta^Td\eta$ in \eqref{eq: general A}. Thus, using Lemma \ref{lemma: ricci with killing reeb}, the flow \eqref{equ: modified ricci flow unit} for $g(t)$ becomes equivalent to the pair:
\begin{align}
    \frac{\partial}{\partial t} \eta &= -\delta^T d \eta+\cdots\label{equ: evol eta 1}\\
    \frac{\partial}{\partial t} g^T &= -2\mathrm{Ric}(g^T)+ \sum_{i=1}^{2n}\iota_{e_i}d\eta \otimes\iota_{e_i}d\eta+\cdots\label{equ: evol ricT}
\end{align}
By the above construction, this flow is clearly equivalent to the Ricci flow of $g(t)$ to highest order, and hence we immediately obtain short time existence and uniqueness from the general theory of Ricci flow. Observe that since $\xi_0$ is Killing and $\eta=\xi_0^\flat$, it follows that $\delta\eta=0$ and hence $-\delta^Td\eta=-\Delta\eta$, where $\Delta$ denotes the transverse Hodge Laplacian. Note also that since $d\eta$ is basic, $\delta^Td\eta\equiv\delta d\eta$ up to lower order terms involving $d\eta$, indeed one can check that $\delta^Td\eta=\delta d\eta-\frac{1}{2}g(d\eta,d\eta) \eta$.
Thus, this is all consistent with the fact that the flow is parabolic (up to gauge-fixing). 

Next we consider the normality condition. Since $\xi_0$ will be fixed from now on, we shall simply denote it by $\xi$. Let us further assume that the initial almost contact metric structure is normal. Recall that this is equivalent to saying that  $M\times S^1$ admits a complex structure defined by \eqref{eqn:cpxstructure}. Since this complex structure is invariant by $\langle\xi,\partial_s=J(\xi)\rangle$, we can consider a sufficiently small open set $U\subset M$ such that the local quotient $U/\langle\xi\rangle$ is a complex manifold with Hermitian metric $g^T$. We  can pullback the holomorphic coordinates on the local quotient so that we have local coordinates $(s,z_1,...,z_n)$ on $M$ in which $\xi=\partial_s$ and
\begin{gather*}
    \eta = ds-\sqrt{-1}h_jdz_j+\sqrt{-1}h_{\bar{j}}dz_{\bar{j}}\\
    g=\eta \otimes \eta + g^T_{i\bar{j}}dz_i\odot d\bar{z}_{\bar{j}}\\
    F= -\sqrt{-1}g^T_{i\bar{j}}dz_i\w d\bar{z}_{\bar{j}}\\
    \varphi=\sqrt{-1}(\partial_{z_j}+\sqrt{-1}h_j\partial_s)\otimes dz_j-\sqrt{-1}(\partial_{\bar{z}_j}+\sqrt{-1}h_{\bar{j}}\partial_s)\otimes d\bar{z}_j
\end{gather*}
where $h,g^T_{ij}$ are local basic functions and $h_j:=\partial_{z_j}h$, see Proposition \ref{prop: local model} and \cite{Pawel2000}. Alternatively, if one wants to work globally then one needs to view this as a complex structure on the quotient bundle $TM/\langle \xi \rangle$, see \cite{Smoczyk2010}.

As before we write $c:=\Lambda_Fd\eta=\tfrac 12 g(d\eta, F)$. We begin by deriving the next curvature identity needed to define the flow equations. 
\begin{lem} \label{lem: 1,1}
Let $(\phi,\eta,\xi,g)$ be a Sasaki with torsion structure. Then on horizontal vector fields
\[
\begin{aligned}
&\mathrm{Ric}(g^T)-\frac14(d^\phi F)^2 -d\eta^2
+\frac12\mathcal L_{\theta^\sharp}g^T
+\frac14\lambda(\phi\cdot,\cdot)=
\left((\rho^B)^{1,1}-c\,d\eta\right)(\phi\cdot,\cdot),
\end{aligned}
\]
where for any differential form $A$ we define the symmetric $2$-tensor
$
A^2(X,Y):=g(\iota_XA,\iota_YA),$
$
\lambda:=-\frac12\Lambda_FdH
\in\Omega^{1,1}_\phi(\mathcal F_\xi^\perp),
$ and $\theta$, as usual, is the transverse Lee form.
\end{lem}

\begin{proof}
By \cite[Proposition 9.1]{FI}, for any horizontal vector fields $x,y$, one has
\[
\begin{aligned}
\frac12\left(
\rho^\nabla(\phi x,y)-\rho^\nabla(x,\phi y)
\right)=
\mathrm{Ric}^\nabla_{\mathrm{Sym}}(x,y)
+\frac12(\mathcal L_{\theta^\sharp}g)(x,y)
+\frac14\lambda(\phi x,y).
\end{aligned}
\]
By Proposition~\ref{prop:rhoB},
\[
\frac12\left(
\rho^\nabla(\phi x,y)-\rho^\nabla(x,\phi y)
\right)
=
\left((\rho^B)^{1,1}-c\,d\eta\right)(\phi x,y).
\]
Moreover,
\[
\mathrm{Ric}^\nabla_{\mathrm{Sym}}
=
\mathrm{Ric}(g)-\frac14H^2.
\]
The conclusion follows from 
Lemma~\ref{lemma: ricci with killing reeb}.
\end{proof}

\subsection{The strong Sasaki with torsion flow}
We now consider the  evolution equations:
\begin{align}
\frac{\partial}{
\partial t
}\eta
&=
-d^{\phi_0} c,\label{eqn:flow1i}\\
\frac{\partial}{
\partial t
}g^T
&=
-2\mathrm{Ric}(g^T)
+2d\eta^2+\frac12(d^\phi F)^2 
-\mathcal L_{\theta^\sharp}g^T
.\label{eqn:flow1ii}
\end{align}
Observe that \eqref{eqn:flow1ii} is of the form \eqref{equ: evol ricT} up to a gauge fixing term. We next explain \eqref{eqn:flow1i}. 
Recall from the above discussion that the quotient bundle $TM/\langle \xi\rangle$ admits a natural complex structure induced by $\varphi_0$, which we shall also denote by $J$ in view of \eqref{eqn:cpxstructure}. We consider variations of $\varphi$ which induce the fixed complex structure $J$ on the quotient bundle i.e.
we only vary $\varphi$ in the vertical direction as in Proposition \ref{prop: normality variation}. Thus, in order to guarantee that the normality condition is preserved under such variations, we need to ensure that $d\eta$ is of type $(1,1)$ along the flow. Equivalently, viewing $M\times S^1$ as a complex manifold, we vary the complex structure on the leaves of foliation generated by $\langle\xi,\partial_s\rangle$
on $M\times S^1$ - this complex structure is determined by $d\eta$ as in \cite{Goldstein2004}. In the above local model, this means that the local coordinates $z_i$ stay $J$-holomorphic. This is similar to the situation considered in \cite{Smoczyk2010}, however, unlike in the Sasaki Ricci flow case (whereby $d\eta=2F$), we also need to prescribe the flow for $\eta$, hence \eqref{eqn:flow1i}. 

Setting
\[
\Phi(g^T)
:=
2\mathrm{Ric}(g^T)-\frac12(d^\phi F)^2 -2 d\eta^2
+\mathcal L_{\theta^\sharp}g^T,
\]
from Lemma~\ref{lem: 1,1}, we see that any Sasaki with torsion structure, the latter is equivalent to
\[
\Phi(g^T)
=
2\left((\rho^B)^{1,1}-c\,d\eta\right)(\phi\cdot,\cdot)
-\frac12\lambda(\phi\cdot,\cdot).
\]
Since the tensors $(\rho^B)^{1,1}$, $d\eta$, and $\lambda$ are all of type $(1,1)$, it follows that $\Phi(g^T)$ is symmetric and $\phi$-invariant. Therefore,
\[
\frac{\partial}{\partial t}g^T=-\Phi(g^T)
\]
is tangent to the space of transverse Hermitian metrics. As $\phi_t$ is reconstructed as the horizontal lift of $J$, the defining algebraic identities of an almost contact metric structure are preserved. 

Since by the normalization discussed above $\xi$ remains a unit Killing vector field for $g(t):=\eta(t) +g^T(t)$ and the complex structure on the quotient is fixed, it follows that $c$ is a basic function. Thus, to prove that the flow \eqref{eqn:flow1i}-\eqref{eqn:flow1ii} is indeed parabolic we need to show that $d^{\varphi_0} c$ is equivalent to $\delta d \eta$ to highest order in agreement with \eqref{equ: evol eta 1}. This can be deduced using the identities in \cite[Theorem 1.1]{Demailly}, but we shall give a more direct proof below.

Since $d\eta$ is a transverse real $(1,1)$-form, we have
\[
*_Td\eta
=
-d\eta\wedge\frac{F^{n-2}}{(n-2)!}
+
(\Lambda_Fd\eta)\frac{F^{n-1}}{(n-1)!}.
\]
Hence, using $c=\Lambda_Fd\eta$, we compute
\[
\begin{aligned}
\delta^Td\eta
&=-*_Td*_Td\eta\\
&=
*_T\left(
d\eta\wedge dF\wedge
\frac{F^{n-3}}{(n-3)!}
\right)
-
*_T\left(
dc\wedge
\frac{F^{n-1}}{(n-1)!}
\right)\\
&\qquad
-
c\,
*_T\left(
dF\wedge
\frac{F^{n-2}}{(n-2)!}
\right).
\end{aligned}
\]
where we used that $F$ is basic since $\xi$ is Killing and $J$ is fixed. With our conventions,
\[
*_T\left(
\beta\wedge
\frac{F^{n-1}}{(n-1)!}
\right)
=
-\phi\beta ,
\]
for any transverse $1$-form $\beta$. Thus,
\[
-*_T\left(
dc\wedge
\frac{F^{n-1}}{(n-1)!}
\right)
=
d^\phi c ,
\]
where $
d^\phi c:=\phi(dc).$ 
Since 
\[
dF\wedge
\frac{F^{n-2}}{(n-2)!}
=
\theta\wedge
\frac{F^{n-1}}{(n-1)!},
\]
we obtain
\[
*_T\left(
dF\wedge
\frac{F^{n-2}}{(n-2)!}
\right)
=
-\phi\theta.
\]
and it follows that
\[
\begin{aligned}
\delta^Td\eta
=
d^\phi c
+
*_T\left(
d\eta\wedge dF\wedge
\frac{F^{n-3}}{(n-3)!}
\right)
+
c\,\phi\theta .
\end{aligned}
\]
Equivalently, writing
\[
d\eta=(d\eta)_0+\frac{c}{n}F,
\]
where $(d\eta)_0$ denotes the primitive part of $d\eta$, we obtain
\begin{equation}\label{eq:codifferential-deta}
\delta^Td\eta
=
d^\phi c
+
*_T\left(
(d\eta)_0\wedge dF\wedge
\frac{F^{n-3}}{(n-3)!}
\right)
+
\frac{2}{n}\,c\,\phi\theta .
\end{equation}

Note that since $c$ is basic and $J$ is fixed, $d^\phi c= d^{\phi_0} c$ and hence
\[
d^{\phi_0} c
=d^\phi c=
\delta^Td\eta+\text{lower-order terms}.
\]
Thus, the system \eqref{eqn:flow1i}-\eqref{eqn:flow1ii} can be expressed as
\[
\begin{split}
\partial_t\eta
&=
-\delta^Td\eta+\text{lower-order terms},\\
\partial_tg^T
&=
-2\mathrm{Ric}(g^T)
-\mathcal L_{\theta^\sharp}g^T
+\text{lower-order terms},
\end{split}
\]
coinciding with \eqref{equ: evol eta 1}-\eqref{equ: evol ricT} up to the gauge transformation.
Thus, one can apply the same strategy as with the Ricci flow by combining the DeTurck vector field with $\theta^\sharp$, see also \cite{fadelflows2}. The resulting gauge-fixed system is parabolic and the usual DeTurck trick pullback argument gives a unique
solution for short time. Alternatively, one can also apply the existence results developed in \cite{Bedulli2018} for transverse parabolic flows. We can summarize the above discussion into:  
\begin{thm}\label{thm:ST-flow}
Given smooth initial data of a Sasaki with torsion structure  $(\phi_0,\eta_0,\xi,g_0)$ on a compact manifold $M$. The evolution equations \eqref{eqn:flow1i}-\eqref{eqn:flow1ii} admit a unique solution $(\phi_t,\eta_t,\xi,g_t)$ for $t\in[0,\epsilon)$ for some $\epsilon>0$ 
which is also a Sasaki with torsion structure.
\end{thm} 
From the definition of $F$ \eqref{def of F}, \eqref{eqn:flow1ii} and  Lemma~\ref{lem: 1,1}, we have
\[
\partial_tF
=
-2\left((\rho^B)^{1,1}-c\,d\eta\right)+\tfrac12 \lambda.
\]
Using the latter together with  \eqref{eqn:flow1i} and the fact that
$dH=dd^\phi F+d\eta\wedge d\eta$, we obtain
\begin{align*}
\partial_t dH &=
dd^\phi(\partial_tF)
+
2d(\partial_t\eta)\wedge d\eta= \tfrac12 dd^\phi \lambda,
\end{align*}
where again we used that $J$ is fixed and $F$ is basic. 
Setting $\Psi:=dH$, the latter can be expressed as
\[
\partial_t\Psi
=
\frac14dd^\phi\Lambda_F\Psi.
\]
Since $\Psi$ is a closed basic form of type $(2,2)$, the identities in
\cite[Theorem~1.1]{Demailly} give

\begin{align*}
dd^\phi\Lambda_F\Psi
&=
d \delta^T \Psi
 + Q*\nabla^T \Psi+\nabla^T Q * \Psi,\\
&= \Delta^T \Psi + Q*\nabla^T \Psi+\nabla^T Q * \Psi,
\end{align*}
where $Q$ correspond to some torsion term and $\nabla^T$ is the transverse Levi-Civita. 
Thus, $\Psi$ satisfies a homogeneous linear
transversely parabolic equation. Since $\Psi_0=0$, uniqueness implies
that $\Psi_t=0$ for as long as the solution exists. Therefore, the strong condition
is preserved along the flow.
\begin{prop}\label{prop:SST-preserved}
Let
$\Phi_t=(\phi_t,\eta_t,\xi,g_t)$
be a solution of \eqref{eqn:flow1i}-\eqref{eqn:flow1ii}. If $\Phi_0$ is strong Sasaki with torsion, then $\Phi_t$ remains strong Sasaki with torsion for as long as the solution exists.
\end{prop}
At this point, we may completely characterize $\nabla$-Einstein structures are static point of the \emph{strong Sasaki with torsion flow}.
\begin{prop}
Let $(M^{2n+1},\phi,\xi,g)$ be a compact strong
Sasaki-with-torsion manifold. Then the structure satisfies 
\begin{equation} \label{eqn:static_point}
d^\phi c=0, \ \ (\rho^B)^{1,1}-c \, d\eta=0
\end{equation}
if and only if it is $\nabla$-Einstein.
\end{prop}
\begin{proof}
The converse implication is clear. Let us prove the direct implication. Assume that $(\phi,\xi,g)$ satisfies $dH=0$ and equation \eqref{eqn:static_point}.
Then, $c$ is constant, and $(\rho^B)^{1,1}=c \, d\eta$. By Remark \ref{rmk:oddeven}, the induced Hermitian structure on the product with $S^1$, satisfies $dH=0$ and the Bismut Ricci form $\tilde \rho^B$ coincides with $\rho^\nabla$. However, 
\[
(\tilde \rho^B)^{(1,1)}= (\rho^B)^{(1,1)}-c d\eta=0.
\]
Therefore, by \cite{Ye}, $\tilde \rho^B=\rho^\nabla$ vanishes identically.
\end{proof}
We now prove that, when the initial condition of the flow \eqref{eqn:flow1i}-\eqref{eqn:flow1ii} is a strong Sasaki-with-torsion structure, the resulting evolution is precisely the generalized Ricci flow \cite{GFStr}, up to gauge equivalence. 
\begin{prop} \label{prop:stationary_sst}
Let $(\phi_t,\eta_t,\xi,g_t)$ be a solution of
\eqref{eqn:flow1i}--\eqref{eqn:flow1ii}.
If the initial structure is strong, then the induced pair $(g_t,H_t)$,
where
\[
g_t=\eta_t^2+g_t^T,
\qquad
H_t=\eta_t\wedge d\eta_t+d^{\phi_t}F_t,
\]
satisfies the gauge-fixed generalized Ricci flow
\[
\partial_tg
=
-2 \Ric(g)+\frac12H^2-\mathcal L_{\theta^\sharp}g,
\qquad
\partial_tH
=
\Delta H-\mathcal L_{\theta^\sharp}H.
\]

\end{prop}

\begin{proof}

By Proposition~\ref{prop:SST-preserved}, the strong condition is preserved along the flow. Therefore, $\lambda=\frac12\Lambda_FdH=0$.
Hence, from the evolution equation for $F$, we have
\[
\partial_tF
=
-2\left((\rho^B)^{1,1}-c\,d\eta\right).
\]
Moreover, since $c$ is basic and the transverse complex structure is fixed,
\[
d^{\phi_0}c=d^\phi c.
\]

By Lemma~\ref{lemma: ricci with killing reeb}, the Ricci tensor of $g$ is
\[
\Ric(g)
=
\frac14g(d\eta,d\eta)\eta\otimes\eta
+\frac12\eta\odot\delta^Td\eta
+\Ric(g^T)
-\frac12d\eta^2.
\]
Since
\[
g=\eta\otimes\eta+g^T,
\]
we have
\[
\begin{aligned}
\partial_tg
&=
(\partial_t\eta)\otimes\eta
+\eta\otimes(\partial_t\eta)
+\partial_tg^T\\
&=
-d^\phi c\otimes\eta
-\eta\otimes d^\phi c
-2\Ric(g^T)
+2d\eta^2
+\frac12(d^\phi F)^2
-\mathcal L_{\theta^\sharp}g^T.
\end{aligned}
\]
We want to prove that
\[
\partial_tg
=
-2\Ric(g)+\frac12H^2-\mathcal L_{\theta^\sharp}g.
\]

Let $x,y$ be horizontal vector fields. From the expression of the Ricci tensor, we have
\[
\Ric(g)(x,y)
=
\Ric(g^T)(x,y)-\frac12d\eta^2(x,y).
\]
Moreover, using that $H=\eta\wedge d\eta+d^\phi F$. For horizontal vector fields $x,y$,
\[
H^2(x,y)
=
2d\eta^2(x,y)+(d^\phi F)^2(x,y).
\]
It follows that
\[
\begin{aligned}
\left(
-2\Ric(g)+\frac12H^2-\mathcal L_{\theta^\sharp}g
\right)(x,y)
&=
-2\Ric(g^T)(x,y)
+2d\eta^2(x,y)
+\frac12(d^\phi F)^2(x,y)
-\mathcal L_{\theta^\sharp}g^T(x,y)\\
&
=
\partial_tg(x,y).
\end{aligned}
\]

We next consider the vertical component. Since $d^\phi c$ is horizontal and $g^T(\xi,\cdot)=0$, we have
\[
\partial_tg(\xi,\xi)=0.
\]
On the other hand,
\[
\Ric(g)(\xi,\xi)
=
\frac14g(d\eta,d\eta).
\]
Since
\[
\iota_\xi H=d\eta,
\]
we also have
\[
H^2(\xi,\xi)=g(d\eta,d\eta).
\]
Moreover, since $\theta^\sharp$ is basic,
\[
[\theta^\sharp,\xi]=0,
\]
and hence
\[
\mathcal L_{\theta^\sharp}g(\xi,\xi)=0.
\]
Consequently,
\[
\left(
-2\Ric(g)+\frac12H^2-\mathcal L_{\theta^\sharp}g
\right)(\xi,\xi)
=
0
=
\partial_tg(\xi,\xi).
\]

It remains to consider the mixed components. Let $x$ be horizontal. From the expression of $\partial_tg$, we have
\[
\partial_tg(x,\xi)
=
-d^\phi c(x).
\]
Recall that
\[
\Ric^\nabla_{\mathrm{Sym}}
=
\Ric(g)-\frac14H^2.
\]
Since $dH=0$, Proposition~9.1 in~\cite{FI}, applied to $x$ and $\xi$, gives
\[
\frac12\rho^\nabla(\phi x,\xi)
=
\Ric^\nabla_{\mathrm{Sym}}(x,\xi)
+\frac12\mathcal L_{\theta^\sharp}g(x,\xi).
\]
Therefore,
\[
\begin{aligned}
\left(
-2\Ric(g)+\frac12H^2-\mathcal L_{\theta^\sharp}g
\right)(x,\xi)
&=
-2\Ric^\nabla_{\mathrm{Sym}}(x,\xi)
-\mathcal L_{\theta^\sharp}g(x,\xi)\\
&
=
-\rho^\nabla(\phi x,\xi).
\end{aligned}
\]
By Proposition~\ref{prop:rhoB},
\[
\rho^\nabla
=
\rho^B-d(c\eta).
\]
Since $\rho^B$ is horizontal, $c$ is basic, and $\iota_\xi d\eta=0$, we obtain
\[
\begin{aligned}
\rho^\nabla(\phi x,\xi)=
-d(c\eta)(\phi x,\xi)=
-dc(\phi x)=d^\phi c(x).
\end{aligned}
\]
Thus,
\[
\left(
-2\Ric(g)+\frac12H^2-\mathcal L_{\theta^\sharp}g
\right)(x,\xi)
=
-d^\phi c(x)
=
\partial_tg(x,\xi).
\]
We have proved that
\[
\partial_tg
=
-2\Ric(g)+\frac12H^2-\mathcal L_{\theta^\sharp}g.
\]

We now prove the evolution equation for the torsion. Decompose
\[
\delta H=B+\eta\wedge a,
\]
where $B$ is a horizontal $2$-form and $a$ is a horizontal $1$-form. By definition,
\[
a=\iota_\xi\delta H.
\]

For a metric connection with closed skew-symmetric torsion, one has
\[
\Ric^\nabla(X,Y)-\Ric^\nabla(Y,X)
=
-\delta H(X,Y).
\]
Since $\nabla\xi=0$, we have
\[
\Ric^\nabla(x,\xi)=0
\]
for every horizontal vector field $x$. Consequently,
\[
a(x)
=
\delta H(\xi,x)
=
-\Ric^\nabla(\xi,x).
\]

Proposition~9.1 in~\cite{FI} gives
\[
\Ric^\nabla(X,\phi Y)
=
\rho^\nabla(X,Y)
+
\left(\nabla_X(c\eta+\phi\theta)\right)(Y).
\]
Applying this identity with $X=\xi$ and $Y=\phi x$, we obtain
\[
-\Ric^\nabla(\xi,x)
=
\rho^\nabla(\xi,\phi x)
+
\left(\nabla_\xi(c\eta+\phi\theta)\right)(\phi x).
\]
Using
\[
\rho^\nabla=\rho^B-d(c\eta),
\]
and the fact that $\rho^B$ is horizontal, we find
\[
\begin{aligned}
\rho^\nabla(\xi,\phi x)
=-d(c\eta)(\xi,\phi x)=
dc(\phi x)=
-d^\phi c(x).
\end{aligned}
\]
Moreover, since $c$ is basic and $\nabla\eta=\nabla\phi=0$,
\[
\left(\nabla_\xi(c\eta+\phi\theta)\right)(\phi x)
=
(\nabla_\xi\theta)(x)= \iota_{\theta^\sharp}d\eta(x).
\]

We now determine the horizontal $2$-form $B$. Let $x,y$ be horizontal vector fields. Applying Proposition~9.1 in~\cite{FI} with $X=x$ and $Y=-\phi y$, we obtain
\[
\Ric^\nabla(x,y)
=
-\rho^\nabla(x,\phi y)-(\nabla_x\theta)(y).
\]
Taking the skew-symmetric part and using
\[
\Ric^\nabla(x,y)-\Ric^\nabla(y,x)
=
-\delta H(x,y),
\]
we find
\[
\begin{aligned}
B(x,y)
&=
\delta H(x,y)=
\rho^\nabla(x,\phi y)
-\rho^\nabla(y,\phi x)
+(\nabla_x\theta)(y)
-(\nabla_y\theta)(x).
\end{aligned}
\]
The horizontal component of $\rho^\nabla$ is
\[
\rho^\nabla_{\lvert\ker\eta}
=
\rho^B-c\,d\eta.
\]
Since $(\rho^B)^{1,1}$ and $c\,d\eta$ are of type $(1,1)$,
\[
\rho^\nabla(x,\phi y)-\rho^\nabla(y,\phi x)
=
2(\rho^B)^{2,0+0,2}(x,\phi y).
\]

Since the torsion of $\nabla$ is $H$, we have
\[
d\theta(x,y)
=
(\nabla_x\theta)(y)
-(\nabla_y\theta)(x)
+
H(x,y,\theta^\sharp).
\]
Since $x,y,\theta^\sharp$ are horizontal,
\[
H(x,y,\theta^\sharp)
=
d^\phi F(x,y,\theta^\sharp)
=
\iota_{\theta^\sharp}d^\phi F(x,y).
\]
It follows that
\[
B+\iota_{\theta^\sharp}d^\phi F
=
d\theta
+
2\left((\rho^B)^{2,0+0,2}\right)_\phi,
\]
where
\[
\left((\rho^B)^{2,0+0,2}\right)_\phi(x,y)
=
(\rho^B)^{2,0+0,2}(x,\phi y).
\]

We now use the fact that the transverse Bismut Ricci form is closed. Using that, we get
\[
\begin{aligned}
d\left((\rho^B)^{2,0+0,2}\right)_\phi
&=
d^\phi(\rho^B)^{1,1}.
\end{aligned}
\]
Since $d^2\theta=0$, differentiating the equation defining $B$, we get
\[
d\left(
B+\iota_{\theta^\sharp}d^\phi F
\right)
=
2d^\phi(\rho^B)^{1,1}.
\]

Since $dH=0$, Cartan's formula gives
\[
\mathcal L_{\theta^\sharp}H
=
d\iota_{\theta^\sharp}H.
\]
Moreover,
\[
\iota_{\theta^\sharp}H
=
-\eta\wedge\iota_{\theta^\sharp}d\eta
+
\iota_{\theta^\sharp}d^\phi F.
\]
Since
\[
\delta H=B+\eta\wedge a,
\]
we have
\[
d\delta H
=
dB+d\eta\wedge a-\eta\wedge da.
\]
Furthermore,
\[
\begin{aligned}
d\iota_{\theta^\sharp}H
&=
-d\eta\wedge\iota_{\theta^\sharp}d\eta
+\eta\wedge d\iota_{\theta^\sharp}d\eta
+d\iota_{\theta^\sharp}d^\phi F.
\end{aligned}
\]
Consequently,
\[
\begin{aligned}
-d\delta H-\mathcal L_{\theta^\sharp}H
&=
-dB-d\eta\wedge a+\eta\wedge da
+d\eta\wedge\iota_{\theta^\sharp}d\eta
-\eta\wedge d\iota_{\theta^\sharp}d\eta
-d\iota_{\theta^\sharp}d^\phi F\\
&=
-d\left(
B+\iota_{\theta^\sharp}d^\phi F
\right)
-d\eta\wedge
\left(
a-\iota_{\theta^\sharp}d\eta
\right)
+\eta\wedge
d\left(
a-\iota_{\theta^\sharp}d\eta
\right).
\end{aligned}
\]
Since
\[
a-\iota_{\theta^\sharp}d\eta
=
-d^\phi c,
\]
we obtain
\[
\begin{aligned}
-d\delta H-\mathcal L_{\theta^\sharp}H
&=
-2d^\phi(\rho^B)^{1,1}
+d\eta\wedge d^\phi c
-\eta\wedge dd^\phi c\\
&=
-2d^\phi(\rho^B)^{1,1}
+d^\phi c\wedge d\eta
-\eta\wedge dd^\phi c.
\end{aligned}
\]

On the other hand, since
\[
H=\eta\wedge d\eta+d^\phi F
\]
and the transverse complex structure is fixed, we have
\[
\partial_t(d^\phi F)
=
d^\phi(\partial_tF).
\]
Therefore,
\[
\begin{aligned}
\partial_tH
&=
\partial_t\eta\wedge d\eta
+\eta\wedge d(\partial_t\eta)
+d^\phi(\partial_tF)\\
&=
-d^\phi c\wedge d\eta
-\eta\wedge dd^\phi c
-2d^\phi\left(
(\rho^B)^{1,1}-c\,d\eta
\right).
\end{aligned}
\]
Since $d\eta$ is closed and of type $(1,1)$,
\[
d^\phi(d\eta)=0.
\]
Hence,
\[
d^\phi(c\,d\eta)
=
d^\phi c\wedge d\eta.
\]
Substituting this identity into the evolution equation for $H$, we find
\[
\begin{aligned}
\partial_tH
&=
-d^\phi c\wedge d\eta
-\eta\wedge dd^\phi c
-2d^\phi(\rho^B)^{1,1}+2d^\phi c\wedge d\eta\\
&=
-2d^\phi(\rho^B)^{1,1}
+d^\phi c\wedge d\eta
-\eta\wedge dd^\phi c.
\end{aligned}
\]
Comparing the last two expressions gives
\[
\partial_tH
=
-d\delta H-\mathcal L_{\theta^\sharp}H.
\]
\end{proof}

\begin{rmk}
Let $\psi_t$ be the family of diffeomorphisms generated by
$\theta_t^\sharp$. Then
$\widetilde g_t:=\psi_t^*g_t,
\widetilde H_t:=\psi_t^*H_t$
satisfy the generalized Ricci flow
\[
\partial_t\widetilde g
=
-2\mathrm{Ric}(\widetilde g)+\frac12\widetilde H^2,
\qquad
\partial_t\widetilde H
=
\Delta_{\widetilde g}\widetilde H.
\]
\end{rmk}
\subsection{Strong Sasaki with torsion flow and pluriclosed flow}

In this section we prove that the Hermitian flow induced on $M\times S^1$ by the strong Sasaki-with-torsion flow is gauge-equivalent to the pluriclosed flow. Note that the solution to strong Sasaki-with-torsion flow on $M$ does not directly lift to the pluriclosed flow on $M\times S^1$ in general; indeed in the later case the complex structure is fixed, whereas in the former case $J$ is evolving.

\begin{thm}
\label{thm:SWT-pluriclosed}
Let $\Phi_t=(\varphi_t,\xi,\eta_t,g_t)$ be a solution of the
Sasaki-with-torsion flow starting from a compact strong
Sasaki-with-torsion manifold $(M, \Phi_0)$, and let
$(\widehat J_t,\widehat\omega_t)$ be the induced Hermitian structure
on $M\times S^1_{\tau}$. Set $
c_t:=\Lambda_{F_t}d\eta_t,$ 
and let $\chi_t$ be the family of diffeomorphisms generated by
$c_t\partial_\tau$ with $\chi_0=\operatorname{Id}$. Then
\[
\chi_t^*\widehat J_t=\widehat J_0
\]
and $\chi_t^*\widehat\omega_t$ is the solution of the pluriclosed flow
on $(M\times S^1,\widehat J_0)$, normalized by
$
\partial_t\omega
=
-2\left(\rho^B_\omega\right)^{1,1},
$
with initial datum $\widehat\omega_0$. In particular, if
$\omega_t^{\mathrm{PC}}$ denotes this solution, then
$
\chi_t^*\widehat\omega_t=\omega_t^{\mathrm{PC}}.
$
\end{thm}
\begin{proof}
Let $(\widehat g_t,\widehat J_t)$ be the Hermitian structure induced on $N=M \times S^1$ by $\Phi_t$. Denote $Z:=\partial_\tau$, where $\tau$ is the coordinate on the $S^1$-circle. The Hermitian structure induced on the product is therefore 
\[
\widehat g_t=d\tau^2+g_t,
\qquad
\widehat\omega_t=d\tau\wedge\eta_t+F_t,
\]
and
\[
\widehat J_t(X+aZ)
=
\varphi_tX-a\xi+\eta_t(X)Z,
\]
in agreement with \eqref{eqn:cpxstructure}.
The Bismut torsion of $(\widehat g_t,\widehat J_t)$ is \[ \widehat H_t = d^c_{\widehat J_t}\widehat\omega_t = \eta_t\wedge d\eta_t+d^{\varphi_t}F_t = H_t. \] 
Since the strong condition is preserved by the Sasaki-with-torsion flow, \[ d\widehat H_t=dH_t=0. \] Thus $\widehat\omega_t$ is pluriclosed with respect to $\widehat J_t$ for every $t$; notice, however, that
$\widehat J_t$ is not fixed.

We next determine the evolution of $\widehat J_t$. Since the
transverse complex structure is fixed, $\partial_t\varphi_t$ takes
values in the Reeb direction, and hence
\[
\partial_t\varphi_t=\xi\otimes dc_t.
\]
Therefore, for every $Y\in \Gamma(TM)$,
\[
(\partial_t\widehat J_t)(Y)
=
dc_t(Y)\xi-d^{\varphi_t}c_t(Y)Z,
\qquad
(\partial_t\widehat J_t)(Z)=0.
\]

Let
\[
X_t:=c_tZ.
\]
All the tensors involved are invariant in the $S^1$-direction, and
$c_t$ is independent of $\tau$. Thus,
\[
[X_t,Y]=-dc_t(Y)Z.
\]
Moreover,
\[
\begin{aligned}
[X_t,\widehat J_tY]
&=
[c_tZ,\varphi_tY+\eta_t(Y)Z]\\
&=
-dc_t(\varphi_tY)Z\\
&=
d^{\varphi_t}c_t(Y)Z.
\end{aligned}
\]
Using $\widehat J_tZ=-\xi$, we obtain
\[
\begin{aligned}
(\mathcal L_{X_t}\widehat J_t)(Y)
&=
[X_t,\widehat J_tY]
-
\widehat J_t[X_t,Y]\\
&=
d^{\varphi_t}c_t(Y)Z-dc_t(Y)\xi.
\end{aligned}
\]
Using that $\partial_t\widehat J_t
+
\mathcal L_{X_t}\widehat J_t (Z)=0$, as $c$ is basic, we get
\[
\partial_t\widehat J_t
+
\mathcal L_{X_t}\widehat J_t
=
0.
\]

We now compute the evolution of $\widehat\omega_t$. Since $dH_t=0$,
the tensor $\lambda_t=\frac12\Lambda_{F_t}dH_t$ vanishes, and the
evolution equation for $F_t$ becomes
\[
\partial_tF_t
=
-2\left((\rho_t^B)^{1,1}-c_t\,d\eta_t\right).
\]
Therefore,
\begin{equation}
\label{eqn:product-omega-evolution}
\partial_t\widehat\omega_t
=
-d\tau\wedge d^{\varphi_t}c_t
-
2\left((\rho_t^B)^{1,1}-c_t\,d\eta_t\right).
\end{equation}

The Bismut connection of
$(\widehat g_t,\widehat J_t)$ splits as
\[
\widehat\nabla_t^B
=
\nabla^{\mathrm{LC}}_{S^1}\oplus\nabla_t.
\]
It follows that its Bismut Ricci form is
\[
\widehat\rho_t^B
=
\rho_t^\nabla
=
\rho_t^B-d(c_t\eta_t),
\]
and hence
\[
\widehat\rho_t^B
=
\rho_t^B-c_t\,d\eta_t-dc_t\wedge\eta_t.
\]

We determine its $(1,1)$-component with respect to
$\widehat J_t$. With our conventions
\begin{equation}
\label{eqn:product-Bismut-Ricci}
\left(\widehat\rho_t^B\right)^{1,1}
=
(\rho_t^B)^{1,1}
-
c_t\,d\eta_t
-
\frac12dc_t\wedge\eta_t
+
\frac12d\tau\wedge d^{\varphi_t}c_t.
\end{equation}

On the other hand,
\[
\mathcal L_{X_t}d\tau
=
d(\iota_{X_t}d\tau)
=
dc_t,
\]
while
\[
\mathcal L_{X_t}\eta_t=0,
\qquad
\mathcal L_{X_t}F_t=0.
\]
It follows that
\[
\mathcal L_{X_t}\widehat\omega_t
=
dc_t\wedge\eta_t.
\]
Combining this identity with
\eqref{eqn:product-omega-evolution}, we obtain
\[
\begin{aligned}
\partial_t\widehat\omega_t
+
\mathcal L_{X_t}\widehat\omega_t
={}&
-2\left((\rho_t^B)^{1,1}-c_t\,d\eta_t\right)
+
dc_t\wedge\eta_t
-
d\tau\wedge d^{\varphi_t}c_t.
\end{aligned}
\]
By \eqref{eqn:product-Bismut-Ricci}, this is precisely
\[
\partial_t\widehat\omega_t
+
\mathcal L_{X_t}\widehat\omega_t
=
-2\left(\widehat\rho_t^B\right)^{1,1}_{\widehat J_t}.
\]

Let $\chi_t$ be the family of diffeomorphisms generated by $X_t$.
Then
\[
\begin{aligned}
\frac{d}{dt}\left(\chi_t^*\widehat J_t\right)
&=
\chi_t^*
\left(
\partial_t\widehat J_t
+
\mathcal L_{X_t}\widehat J_t
\right)=0.
\end{aligned}
\]
Since $\chi_0=\operatorname{Id}$, we conclude that
\[
\chi_t^*\widehat J_t=\widehat J_0.
\]

Set
\[
\overline\omega_t:=\chi_t^*\widehat\omega_t.
\]
Using the invariance of the Bismut connection by holomorphic diffeomorphisms, we find
\[
\begin{aligned}
\partial_t\overline\omega_t
&=
\chi_t^*
\left(
\partial_t\widehat\omega_t
+
\mathcal L_{X_t}\widehat\omega_t
\right)=
-2\chi_t^*
\left(
\left(\widehat\rho_t^B\right)^{1,1}_{\widehat J_t}
\right)=
-2\left(
\rho^B_{\overline\omega_t}
\right)^{1,1}_{\widehat J_0}.
\end{aligned}
\]
Moreover,
\[
\overline\omega_0=\widehat\omega_0.
\]
Thus $\overline\omega_t$ solves the pluriclosed flow on the fixed
complex manifold $(M\times S^1,\widehat J_0)$ with initial datum
$\widehat\omega_0$. By uniqueness,
\[
\chi_t^*\widehat\omega_t
=
\omega_t^{\mathrm{PC}}.
\]
 \end{proof}

\bibliographystyle{plain} 
\bibliography{biblio} 
\noindent
\small 
B. Brienza: IMPA, Estrada Dona Castorina, 110, CEP 22460-320, Rio de Janeiro, RJ (Brasil) \\
Email: \texttt{brienza.beatrice@impa.br}
\vskip0.5em
\noindent
A. Fino: Dipartimento di Matematica ``G. Peano'', Universit\`{a} degli studi di Torino, Via Carlo Alberto 10, Torino (Italy) \\
Department of Mathematics and Statistics, Florida International University, Miami, FL 33199, (USA) \\
Email: \texttt{annamaria.fino@unito.it} \, \, \texttt{afino@fiu.edu}
\vskip0.5em
\noindent
U. Fowdar: Institute of Mathematics, University of Warsaw, Banacha 2, 02-097 Warszawa (Poland)\\
Email: \texttt{u.fowdar@uw.edu.pl}
\vskip0.5em
\noindent
G. Grantcharov: Florida International University, Miami, FL 33199, (USA) \\
Institute of Mathematics and Informatics Bulgarian Academy of Sciences, \\
Acad. Georgi Bonchev str. 8, 1113 Sofia, (Bulgaria)\\
Email: \texttt{grantchg@fiu.edu}\, \,\texttt{grantchg@math.bas.bg}

\end{document}